\documentclass{amsart}

\usepackage{amssymb,amsfonts, amsmath, amsthm}
\usepackage[all,arc]{xy}
\usepackage{enumerate}
\usepackage{mathrsfs}
\usepackage{mathtools}
\usepackage{stmaryrd}
\usepackage[mathscr]{euscript}
\usepackage{array}
\usepackage{graphicx}
\usepackage{tikz}
\usepackage{tikz-cd}
\usepackage{mathpazo} 
\usepackage{textcomp}
\usepackage{bbold}
\usepackage[bbgreekl]{mathbbol}
\usepackage{enumitem}

\usepackage[pagebackref, colorlinks=true, linkcolor=blue, citecolor = red]{hyperref}
\renewcommand*{\backref}[1]{}
\renewcommand*{\backrefalt}[4]{({%
		\ifcase #1 Not cited.%
		\or On p.~#2%
		\else On pp.~#2%
		\fi%
	})}

\usepackage[nameinlink,capitalise,noabbrev]{cleveref}
\crefname{subsection}{Subsection}{Subsection}

\usetikzlibrary{patterns}

\usepackage{geometry}
\makeatletter
\newcommand{\setword}[2]{%
  \phantomsection
  #1\def\@currentlabel{\unexpanded{#1}}\label{#2}%
}
\makeatother

\DeclareMathAlphabet{\mathbbe}{U}{bbold}{m}{n}

\newcommand{\A}{\mathscr{A}}
\newcommand{\C}{\mathscr{C}}
\newcommand{\cC}{\mathcal{C}}
\newcommand{\D}{\mathscr{D}}
\newcommand{\F}{\mathscr{F}}
\newcommand{\cF}{\mathcal{F}}
\newcommand{\bF}{\mathbb{F}}
\newcommand{\bG}{\mathbb{G}}
\newcommand{\cL}{\mathcal{L}}
\newcommand{\bN}{\mathbb{N}}
\newcommand{\s}{\mathscr{S}}
\newcommand{\cM}{\mathcal{M}}
\newcommand{\bM}{\mathbb{M}}
\newcommand{\cS}{\mathcal{S}}
\newcommand{\bT}{\mathbb{T}}
\newcommand{\cT}{\mathcal{T}}
\newcommand{\cU}{\mathcal{U}}
\newcommand{\cV}{\mathcal{V}}
\newcommand{\cW}{\mathcal{W}}
\newcommand{\cP}{\mathcal{P}}
\newcommand{\bX}{\mathbb{X}}
\newcommand{\bY}{\mathbb{Y}}

\newcommand{\set}{\mathscr{S}\mathrm{et}}
\newcommand{\sset}{s\set}
\newcommand{\cat}{\mathscr{C}\mathrm{at}}
\newcommand{\Fun}{\mathrm{Fun}}
\newcommand{\Map}{\mathrm{Map}}
\newcommand{\Ho}{\mathrm{Ho}}
\newcommand{\id}{\mathrm{id}}
\newcommand{\Alg}{\mathcal{A}\mathrm{lg}}
\newcommand{\Algf}{\mathcal{A}\mathrm{lg}_{\textbf{\text{f}}}}
\newcommand{\Fib}{\mathcal{F}\mathrm{ib}}
\newcommand{\bFib}{\mathbb{F}\mathrm{ib}}
\newcommand{\Inst}{\mathcal{I}\mathrm{nst}}
\newcommand{\Mnd}{\mathcal{M}\mathrm{nd}}
\newcommand{\Sub}{\mathscr{S}\mathrm{ub}}
\newcommand{\End}{\mathcal{E}\mathrm{nd}}
\newcommand{\Grpd}{\mathscr{G}\mathrm{rpd}}
\newcommand{\Ty}{\mathbb{T}\mathrm{y}}
\newcommand{\Tm}{\mathbb{T}\mathrm{m}}
\newcommand{\Eq}{\mathrm{Eq}}
\newcommand{\idtoequiv}{\mathrm{idtoequiv}}
\newcommand{\colim}{\mathrm{colim}}

\newtheorem{theorem}[equation]{Theorem}
\newtheorem{lemma}[equation]{Lemma}
\newtheorem{proposition}[equation]{Proposition}
\newtheorem{corollary}[equation]{Corollary}

\theoremstyle{definition}
\newtheorem{definition}[equation]{Definition}
\newtheorem{example}[equation]{Example}
\theoremstyle{remark}
\newtheorem{remark}[equation]{Remark}
\newtheorem{notation}[equation]{Notation}

\numberwithin{equation}{section}

\title{Non-Standard Models of Homotopy Type Theory}

\date{August 2025}

\author{Nima Rasekh}
\address{Institut f{\"u}r Mathematik und Informatik, Universit{\"a}t Greifswald, Greifswald, Germany}
\email{nima.rasekh@uni-greifswald.de}

\subjclass[2020]{03B38, 03H05, 18N40, 18C50, 03G30, 18N60}
\keywords{Homotopy type theory, categorical semantics, model category theory, filter quotient construction, non-standard models}

\begin{document}

\begin{abstract}
	\emph{Homotopy type theory} is a modern foundation for mathematics that introduces the \emph{univalence axiom} and is particularly suitable for the study of homotopical mathematics and its formalization via proof assistants. In order to better comprehend the mathematical implications of homotopy type theory, a variety of models have been constructed and studied. Here a model is understood as a model category with suitable properties implementing the various type theoretical constructors and axioms. A first example is the simplicial model due to Kapulkin--Lumsdaine--Voevodsky \cite{kapulkinlumsdaine2021kanunivalent}. By now, many other models have been constructed, due to work of Arndt, Kapulkin, Lumsdaine, Warren and particularly Shulman, culminating in a proof that every Grothendieck $\infty$-topos can be obtained as the underlying $\infty$-category of a model category that models homotopy type theory \cite{arndtkapulkin2011modelstypetheory,shulman2015homotopycanonicity,lumsdainewarren2015localuniverses,lumsdaineshulman2020goodexcellent,shulman2019inftytoposunivalent}. 
	
	In this paper, we propose the filter quotient construction as a new method to construct further models of homotopy type theory. Concretely, we prove that with minor assumptions, the filter quotient construction \cite{rasekh2025filtermodelcat} preserves all model categorical properties individually that implement various type theoretical constructors and axioms. On the other hand, the filter quotient construction does not preserve many external properties that are of set theoretical nature, such as cocompleteness, local presentability or cofibrant generation. Combining these, the filter quotient construction preserves models of homotopy type theory and can result in models that have not been considered before and exhibit behaviors that diverge from any of the established models. 
\end{abstract}

\maketitle

\section{Introduction}

\subsection{From Homotopy Theory to Homotopy Type Theory}
Homotopy theory is a branch of mathematics that studies mathematical structures (topological spaces, chain complexes, categories, spectra, ...) up to some chosen notion of equivalence (homotopical equivalences, quasi-isomorphism, categorical equivalence, stable equivalence, ...). Historically, it first arose in topology, where topological spaces were considered up to (weak) homotopy equivalences. From these early days, homotopy theory faced a fundamental tension: homotopy equivalences of topological spaces are strictly weaker than bijections of the underlying sets. This means homotopy equivalences are not naturally part of the underlying mathematical foundation. As a result, first steps in algebraic topology and homotopy theory added those equivalences artificially to the structure, for example via localizations or via the construction of homotopy categories \cite{brown1965abstracthomotopy,gabrielzisman1967calculusfractions}. This situation suggested the need for pursuing and developing a foundation for mathematics more suitable for homotopical mathematics.

Since the 2010s, a new foundation has been developed, that can properly incorporate homotopical mathematics: \emph{homotopy type theory} \cite{hottbook2013}. In general, type theories are a foundation for mathematics with a particular focus syntax. This syntactic angle makes it a particularly suitable foundation for formalizing mathematics via proof assistants, such as Rocq or Lean \cite{coquandhuet1986coc}. Homotopy type theory is one such type theory that, via the \emph{univalence axiom}, axiomatically incorporates a notion of equivalence, making it the ideal setting for homotopy theory. As a result, in recent years many homotopical definitions and theorems have been developed in homotopy type theory, what is known as \emph{synthetic homotopy theory} or \emph{synthetic topology}. 

As already pointed out, the syntactical nature of type theories makes homotopy type theory also a suitable framework for the formalization of homotopical mathematics via proof assistants. Prominent examples includes the formalization of various results in topology (such as the Blakers-Massey theorem or the computation of homotopy groups) in proof assistants, using \emph{(cubical) Agda} and \emph{Rocq} \cite{coquandhuet1986coc,licatashulman2013fundamental,favoniafinsterlicatalumsdaine2016blakersmassey,vezzosimortbergabel2019cubicalagda,brunerieljungstrommortberg2022synthetic,ljungstrommortberg2023pi4s3}.

\subsection{Models of homotopy type theory}
While a type theoretical approach is ideal when working with computers and proof assistants, many mathematicians prefer a semantic approach. This has motivated the development of rigorous methods to translate between type theories and type theoretical constructors on the one side (syntax), and categorical structures with categorical universal properties on the other side (semantics). This is known as the \emph{syntax-semantics duality}. From this perspective, for a given type theory we construct categories that are the \emph{models of the type theory}, whereas on the other side, for a given category, we extract the type theory as the \emph{internal language} of the category. As a result there is now an extensive literature establishing precise equivalences between various type theories (such as Martin-L{\"o}f type theories, possibly with $\Pi$-types) and various categories (such as locally Cartesian closed categories or topoi) \cite{seely1983hyperdoctrines,seely1984locallycartesian,lambekscott1988higherorderlogic,hofmann1995lccc,curiengarnerhofmann2014categorical,clairambaultdybjer2014biequivalence}. This effort, in particular, includes a model of Martin-L{\"o}f type theory via groupoids, which is the first example where non-trivial equivalences (in this case for groupoids) arise \cite{hofmannstreicher1998groupoidtypes}.

Unfortunately, this translation between type theories and categorical models is not straightforward, and involves many technicalities. As a result many technical categorical tools have been introduced that can assist with this translation effort, and give us effective ways to construct categorical models out of type theories. This includes \emph{contextual categories} and \emph{categories with attributes} \cite{cartmell1986contexualcategories}, \emph{display map categories} \cite{taylor1987displaymap}, \emph{comprehension categories} \cite{jacobs1993comprehension}, \emph{categories with families} \cite{dybjer1996families}, and \emph{natural models} and \emph{natural pseudo-models} \cite{awodey2018natural}. Moreover, many of these approaches have proven to be equivalent in recent years \cite{kapulkinlumsdaine2018typetheories,ahrenslumsdainenorth2025semanticframeworks}. Each one of these different notions is defined as a category along with additional data, that captures the data that comes from the type theory. Using these precise technical tools, many further categorical models of various type theories have been constructed in recent years.

\subsection{Semantics via Model Categories}
The development outlined in the previous section suggests a natural solution when trying to better understand homotopy type theory: we pick one of the technical categorical methods for constructing models of type theories and apply it to the specific case of homotopy type theory. Unfortunately, the notion of equivalence that is inherent to homotopy type theory gives it a homotopical flavor, which suggests that any suitable model should also involve a notion of equivalence, making the construction and study of models of homotopy type theory particularly challenging. In the particular case of the groupoid model this issue was addressed ``by hand'', as natural equivalences that are equivalent are necessarily equal, naturally restricting the amount of coherences that needed to be addressed \cite{hofmannstreicher1998groupoidtypes}. However, for the general case, the required coherences to properly deal with equivalences are not restricted in any way. What is hence necessary is a strict category with a manageable notion of equivalence. Fortunately, here one can benefit from the pioneering work of Quillen.

Quillen originally introduced the notion of a \emph{model category}\footnote{The ``model'' in model category is not related to models of type theory, and arose independently as a historical coincidence.} as a category with an axiomatic notion of weak equivalence along with further structures (fibrations and cofibrations) that help effectively manage the behavior of weak equivalences in a model category \cite{quillen1967modelcats}. Model categories have found many applications in mainstream homotopy theory. As a result there is now an extensive literature permitting an effective construction and study of a variety of model categorical notions \cite{hovey1999modelcategories,hirschhorn2003modelcategories}. Beyond these advances, we now also have a very good understanding of the relation between model category theory and other categorical axiomatizations of the notion of equivalence, namely \emph{$\infty$-categories}. In particular, every model category comes with an \emph{underlying $\infty$-category}, which retains many important properties of the model category \cite{lurie2009htt}.

The combination of these two facts, namely the existence of additional challenges when pursuing models of homotopy type theory and the powerful ability of model categories to capture homotopical data, motivated several figures, such as Voevodsky and Awodey, to pursue an alternative intermediate step when constructing categorical models of type theories: instead of directly constructing a model for homotopy type theory, we first construct a model category, and then extract a model for homotopy type theory from it, using one of the technical categorical models discussed above \cite{awodeywarren2009identitytypes,awodeygarnermartinlofvoevodsky2011univalence,voevodsky2014origins}.

This approach was tremendously successful. The first success was the construction of the simplicial model of homotopy type theory, where the authors prove that the \emph{Kan model structure} on the category of simplicial sets (whose underlying $\infty$-category is the $\infty$-category of spaces), results in a contextual category that is a model for homotopy type theory \cite{kapulkinlumsdaine2021kanunivalent}. After that first success, building on several major developments due to Arndt, Kapulkin, Lumsdaine, Shulman and Warren \cite{arndtkapulkin2011modelstypetheory,shulman2015homotopycanonicity,shulman2015elegantunivalence,lumsdainewarren2015localuniverses,shulman2017eidiagrams,lumsdaineshulman2020goodexcellent}, Shulman proved that \emph{type-theoretic model topoi}, whose underlying $\infty$-categories recover all \emph{Grothendieck $\infty$-topoi}, give us a natural pseudo-model that indeed models homotopy type theory with all relevant axioms and constructors, including univalent universes \cite{shulman2019inftytoposunivalent}.\footnote{The model categorical approach has also been used in other contexts, such as \emph{cubical type theory} \cite{awodeycavallocoquandriehlsattler2024equivariant}. Moreover, there are further developments towards homotopical models of type theories, for example via \emph{weak algebraic factorization systems} \cite{gambinolarre2023models}} While all these results represent a significant step towards the classification of all models of homotopy type theory, many models remain unexplored, and concretely any model for which the underlying $\infty$-category is not a Grothendieck $\infty$-topos.

\subsection{New Semantics via Filter Quotients}
In this work we introduce a new method for constructing models of homotopy type theory via \emph{filter quotient model categories}. This results in models that heretofore had not been constructed and differ in their properties from the established models. The idea of using filters in mathematical logic to construct non-trivial models goes back to major figures such as Skolem \cite{skolem1934earlyultraproducts} and {\L}o{\'s} \cite{los1955ultraproduct}. Additionally, while not explicitly in the proof, an intuition regarding filter quotients plays an important role in Cohen's forcing construction, proving the independence of the continuum hypothesis \cite{cohen1963forcing}. This motivated Lawvere and Tierney to introduce topos theory as a categorical foundation for mathematics and to define therein \emph{filter quotients} as a way to construct new models \cite{lawvere1964elementarysets,tierney1972elementarycontinuum}. These techniques were further refined by Adelman and Johnstone \cite{adelmanjohnstone1982serreclasses}, and have since become a standard part of modern topos theory \cite{johnstone1977topostheory,maclanemoerdijk1994topos}. Recently, the filter quotient construction has also been generalized to $\infty$-categories \cite{rasekh2021filterquotient} and model categories \cite{rasekh2025filtermodelcat}. 

In order to understand how filter quotients help construct models of homotopy type theory, it is instructive to first review relevant aspects of the filter quotient construction. A filter is defined as a certain subset of a poset. Given a model category $\cM$ and a suitable filter in the poset of subterminal objects in $\cM$, denoted $\Phi$, we obtain a new model category $\cM_\Phi$, along with a projection functor $P_\Phi\colon \cM \to \cM_\Phi$ \cite{rasekh2025filtermodelcat}. With this result at hand, the major aim of this paper is to state and prove a precise version of the following meta-statement:
\begin{align} \label{slogan}
 \text{\emph{\textbf{\parbox{13.5cm}{The functor $P_\Phi$ preserves enough model categorical properties, so that the filter quotient model category $\cM_\Phi$ models the same type theoretic constructors and axioms as $\cM$.}}}}
\end{align}
Here the word ``enough'' is important, as the projection functor will certainly not preserve all properties, such as infinite (co)limits and cofibrantly generated model structures, necessitating a careful analysis. We tackle this problem in three steps:

\begin{enumerate}[leftmargin=*]
	\item In \cref{not:main} we summarize all relevant model categorical properties required to model type theoretical constructors via model categories.
	\item In \cref{thm:prop} we precisely prove which one of the properties are preserved by the projection functor.
	\item In \cref{thm:main} we prove that the preservation properties in \cref{thm:prop} suffice to model the same type theoretic constructors.
\end{enumerate}

\subsection{New Models and Their Implications}
Establishing this result immediately implies that anything proven in homotopy type theory is valid in more situations than previously anticipated. This includes the examples already mentioned, but also other results, such as the \emph{Hurewicz theorem} \cite{christensenscoccola2023hurewicz}. Beyond this general observation, this approach has concrete implications regarding our understanding of homotopy type theory.
\begin{enumerate}[leftmargin=*]
	\item We can immediately observe that the universal property of infinite coproducts cannot be expressed in the internal language of homotopy type theory (\cref{cor:independent infinite colimit}).
	\item We can observe that while all models of homotopy type theory have a natural number object, the size of the set of natural numbers can vary significantly between models, and models of homotopy type theory can have \emph{non-standard natural numbers} (\cref{cor:independent standard nno}).
	\item We can observe the existence of models of homotopy type theory that are well-pointed, meaning generated by the terminal object, but not equivalent to the $\infty$-category of spaces, which was an impossibility with previously known models (\cref{prop:wellpointed unique}). Motivated by other constructions in logic, such as \cite{palmgren1997nonstandardanalysis}, such a model can reasonably be considered a \emph{non-standard model of homotopy type theory} or a \emph{non-standard model for spaces} (\cref{ex:nonstandard model}).
\end{enumerate}
Fundamentally all of these examples are a concrete manifestation of the fact that in certain models a notion of ``infinity'' internal to homotopy type theory diverges from a notion of ``infinity'' from the model. This observation is in fact crucial when studying $\infty$-category theory in different foundations (i.e. internal to homotopy type theory and $\infty$-topos theory), which is the aim of \cite{rasekh2025shott}. See \cref{rem:infinity} for a detailed discussion.

\subsection{Acknowledgment}
I would like to thank Niels van der Weide for his careful reading and helpful comments on the first draft of this paper. I am also grateful to the Max Planck Institute for Mathematics in Bonn for its hospitality and financial support. Moreover, I am grateful to the Hausdorff Research Institute for Mathematics in Bonn, Germany, for organizing the trimester ``Prospects of Formal Mathematics,'' funded by the Deutsche Forschungsgemeinschaft (DFG, German Research Foundation) under Germany's Excellence Strategy – EXC-2047/1 – 390685813, which resulted in many fruitful interactions and relevant conversations regarding foundation and type theory.

\section{Filter Quotient Models of Homotopy	Type Theory}\label{sec:main}
In this section we review the relevant background (regarding model categories and filter quotients) and present the main results. The proofs have been relegated to \cref{sec:technicalities}.

\subsection{Model Categories and Filter Quotients} \label{subsec:model cat}
In this short subsection we review model categorical terminology and how they relate to filter quotients. 

\begin{definition}
	A model structure on a category $\C$ is a triple $(\cF,\cC,\cW)$ of classes of morphisms in $\cM$ satisfying the following axioms:
	\begin{itemize}[leftmargin=*]
		\item $(\cC \cap \cW, \cF)$ and $(\cC, \cF \cap \cW)$ are weak factorization systems.
		\item If two of $f$, $g$, and $g \circ f$ are in $\cW$, so is the third.
	\end{itemize}
	A model category is a tuple $(\C,\cM)$ of a category $\C$, that is either finitely (co)complete or small (co)complete, and $\cM$ is a model structure on $\C$. 
\end{definition}

\begin{remark}
	The first definition of a model category, due to Quillen, only assumed finite (co)limits \cite{quillen1967modelcats}. However, more modern references assume small (co)limits \cite{hovey1999modelcategories,hirschhorn2003modelcategories}. In the next subsection we will see examples of both, hence we will not assume either and always mention whether a model category is finitely or small (co)complete.
\end{remark}

We now review how to construct new model structures via filter quotients.

\begin{definition}
	Let $(P, \leq)$ be a poset. A \emph{filter} on $P$ is a non-empty subset of $P$, which is:
	\begin{itemize}[leftmargin=*]
		\item Upwards closed: for all $x \in \Phi$ and $y \geq x$, $y \in \Phi$.
		\item Intersection closed: for all $x,y$ in $\Phi$, there exists a $z$ in $\Phi$, with $z \leq x, y$.
	\end{itemize}
\end{definition}

Recall that an object $U$ in a category $\C$ is subterminal, if for every object $X$ in $\C$, there is at most one morphism from $X$ into $U$. 

\begin{definition}
	Let $\C$ be a category. A \emph{filter of subterminal objects} on $\C$ is a filter on the poset of subterminal objects $\Sub(\C)$.
\end{definition}

\begin{definition} \label{def:filter quotient category}
 Let $\C$ be a category with finite products and $\Phi$ a filter of subterminal objects. Then the filter quotient $\C_\Phi$ is a category with the same objects as $\C$, and for two objects $X, Y$ 

	\[\C_\Phi(X,Y)=\left(\coprod_{U\in\Phi}\C(X\times U,Y)\right)\mathbin{\big/\mkern-6mu\sim_\Phi}\]
	where $f\colon U \times X \to Y$, $g\colon V \times X \to Y$ are equivalent if there exists a $W \in \Phi$, with $W \leq U, V$, and $f$ and $g$ are equal when restricted to $W$.
\end{definition}

There is an evident \emph{projection functor} $P_\Phi\colon \C \to \C_\Phi$, which is the identity on objects and sends each morphism to its equivalence class, and preserves many properties of interest.

\begin{proposition} \label{prop:filter quotient category}
	Let $\C$ be a category with finite products, $\Phi$ a filter of subterminal objects.
	\begin{enumerate}[leftmargin=*]
		\item $P_\Phi$ preserves finite (co)limits, monomorphisms, (local) exponentiability, subobject classifiers, and natural number objects.
		\item If $\C$ is (locally) Cartesian closed, an elementary topos, or simplicially enriched, then so is $\C_\Phi$.
		\item If $\C$ is a simplicial filter, meaning $U$ in $\Phi$, $K$ a finite simplicial set and $X$ in $\cM$, we have $(X \otimes K) \times U = (X \times U) \otimes K$, and $\C$ is (co)tensored over finite simplicial sets, then $\C_\Phi$ is (co)tensored over finite simplicial sets.
	\end{enumerate} 
\end{proposition}

\begin{proof}
	The first item is proven in \cite[Example A.2.1.13]{johnstone2002elephanti}, see also \cite[Section 9.4]{johnstone1977topos}. Except for the simplicial enrichment, the second item follows directly from the first item.  The simplicial enrichment and co(tensor) is proven in \cite[Lemma 4.12]{rasekh2025filtermodelcat}. 
\end{proof}

In general this definition does not interact well with model structures, necessitating adjusting the definition. Following \cite[Definition 3.9]{rasekh2025filtermodelcat}, for a given model category $\cM$ and filter of subterminal objects $\Phi$, a class of morphisms $S$ in $\cM$ is $\Phi$-product stable, if for every $f$ in $S$ and $U$ in $\Phi$, $f \times U$ is in $S$.

\begin{definition}
	Let $\cM$ be a model category. A \emph{model filter} $\Phi$ on $\cM$ is a filter of subterminal objects with the following properties:
	\begin{enumerate}[leftmargin=*]
		\item Every $U$ in $\Phi$ is fibrant.
		\item The cofibrations and weak equivalences in $\cM$ are $\Phi$-product stable.
		\end{enumerate}
		Additionally, $\Phi$ is called a \emph{simplicial model filter } if it is a model filter that is also simplicial.
\end{definition}

With this definition we do have the expected result. 

\begin{notation}
	Let $\C$ be a category with finite products and $\Phi$ a filter of subterminal objects. Let $S$ be a $\Phi$-product stable set of morphisms. Let $S_\Phi$ be the set of morphisms in $\C_\Phi$ with the property that $f \in S_\Phi$ if there exists $U \in \Phi$ such that $f \times U \in S$.
\end{notation}

\begin{theorem}[{\cite[Theorem 3.16, Corollary 4.17]{rasekh2025filtermodelcat}}] \label{thm:filter quotient model structure}
	Let $\cM$ be a model category and $\Phi$ a model filter on $\cM$. 
\begin{enumerate}[leftmargin=*]
	\item The filter quotient $\cM_\Phi$ carries a model structure given by $(\cF_\Phi,\cC_\Phi,\cW_\Phi)$. In particular, $P_\Phi\colon\cM \to \cM_\Phi$ preserves fibrations, cofibrations, weak equivalences, and right properness.
	\item If $\Phi$ is a simplicial model filter, then $\cM_\Phi$ is a simplicial model category, and $P_\Phi$ preserves the simplicial enrichment. 
\end{enumerate}
\end{theorem}

Finally, let us observe one particularly relevant example: filter products.

\begin{example} \label{ex:filter product}
	Let $\cM$ be a model category and $I$ a set and $\Phi$ a filter of subsets of $I$. The \emph{filter product} $(\prod_I \cM)_\Phi$ is the filter quotient of $\prod_I\cM$, and usually denoted $\prod_\Phi\cM$. By \cite[Lemma 6.3]{rasekh2025filtermodelcat}, if $\cM$ has a strict initial object, then, $\prod_\Phi \cM$ carries the filter product model structure, with the same properties as in \cref{thm:filter quotient model structure}. In particular, by \cite[Corollary 6.10]{rasekh2025filtermodelcat}, if $\cM$ is a simplicial model category, then $\Phi$ is a simplicial model filter, and $\prod_\Phi \cM$ is a simplicial model category.
\end{example}

Let us end with an example of the example. For this next example recall that a filter on a set $I$ is \emph{principal} if it is of the form $\{J \subseteq I | x \in J\}$ for some fixed $x \in I$, and \emph{non-principal} otherwise.

\begin{example} \label{ex: filter product n}
	Let $\cM$ be a model category with strict initial object. Let $\bN$ be the set of natural numbers, and $\Phi$ a non-principal filter on $\bN$. Then the filter product $\prod_\Phi \cM$ is a model category without infinite coproducts, and hence in particular also not locally presentable or cofibrantly generated. See \cite[Example D.5.1.7]{johnstone2002elephantii} for a more detailed discussion.
\end{example}

\subsection{Filter Quotients model Homotopy Type Theory} \label{subsec:main}
We are now ready to state the main results and realize the vision outlined in \ref{slogan}. First of all we list all relevant model categorical properties considered in the literature \cite{arndtkapulkin2011modelstypetheory,shulman2015homotopycanonicity,lumsdaineshulman2020goodexcellent,shulman2019inftytoposunivalent}.

\begin{notation} \label{not:main}
	We use the following notational conventions regarding model categorical properties.
	\begin{itemize}[leftmargin=*]
		  \item \hypertarget{item:flc}{} \textbf{FLC}: The underlying category has \textbf{F}inite \textbf{L}imits and \textbf{C}olimits.
				\item \hypertarget{item:lc}{} \textbf{LC}: The underlying category has small \textbf{L}imits and \textbf{C}olimits.
				\item \hypertarget{item:gt}{} \textbf{GT}: The underlying category is a \textbf{G}rothendieck \textbf{T}opos\footnote{A \emph{Grothendieck topos} is a left-exact localization of a presheaf category. For more details see \cite[Section 3]{maclanemoerdijk1994topos}.}.
				\item \hypertarget{item:et}{} \textbf{ET}: The underlying category is an \textbf{E}lementary \textbf{T}opos\footnote{An \emph{elementary topos} is a locally Cartesian closed category with subobject classifier. For more details see \cite[Section 4]{maclanemoerdijk1994topos}.}.
				\item \hypertarget{item:lp}{} \textbf{LP}:	The underlying category is \textbf{L}ocally \textbf{P}resentable\footnote{A \emph{locally presentable category} is an accessible localization of a presheaf category. For more details see \cite{adamekrosicky1994presentable}.}.
				\item \hypertarget{item:lcc}{} \textbf{LCC}: The underlying category is \textbf{L}ocally \textbf{C}artesian \textbf{C}losed\footnote{A finitely complete category $\C$ is \emph{locally Cartesian closed category} if for all $f\colon X \to Y$ the pullback functor $f^*\colon\C_{/Y} \to \C_{/X}$ has a right adjoint.}.
				\item \hypertarget{item:slc}{} \textbf{SLC}: The	underlying category is \textbf{S}implicially \textbf{L}ocally \textbf{C}artesian closed\footnote{A locally Cartesian closed category $\C$ is \emph{simplicially locally Cartesian closed} if $f^*$ is a simplicially enriched left adjoint.}. 
				\item \hypertarget{item:rp}{} \textbf{RP}: The model structure is \textbf{R}ight \textbf{P}roper.
				\item \hypertarget{item:cg}{} \textbf{CG}: The model structure is \textbf{C}ofibrantly \textbf{G}enerated\footnote{A model structure is \emph{cofibrantly generated} if there is a set of (trivial) cofibrations generating all cofibrations and hence (via lifting property) determining all (trivial) fibrations. For a precise definition see \cite[Section 2.1]{hovey1999modelcategories}}.
				\item \hypertarget{item:s}{} \textbf{S}: The model structure is \textbf{S}implicial.
				\item \hypertarget{item:cim}{} \textbf{CIM}: The \textbf{C}ofibrations in the model structure \textbf{I}nclude the \textbf{M}onomorphisms.
				\item \hypertarget{item:cem}{} \textbf{CEM}: The \textbf{C}ofibrations in the model structure are \textbf{E}qual to the \textbf{M}onomorphisms.
				\item \hypertarget{item:fe}{} \textbf{FE}: \textbf{F}ibrations in the model structure are \textbf{E}xponentiable\footnote{A fibration $p\colon Y \to X$ in $\cM$ is \emph{exponentiable}, if $p^*\colon\cM_{/X} \to \cM_{/Y}$ has a right adjoint.}.
				\item \hypertarget{item:tcp}{} \textbf{TCP}: \textbf{T}rivial \textbf{C}ofibrations in the model structure are closed under \textbf{P}ullbacks along fibrations.
				\item \hypertarget{item:cl}{} \textbf{CL}: \textbf{C}ofibrations in the model category are stable under all \textbf{L}imits that exist.
				\item \hypertarget{item:f}{} \textbf{F}: There is a locally representable and relatively acyclic notion of \textbf{F}ibred structure $\bF$ such that $|\bF|$ is the class of all fibrations.\footnote{Defined in \cite[Definition 3.1, Definition 3.10, Definition 5.11]{shulman2019inftytoposunivalent}. See \cref{def:lparanofscaf} for a more detailed review.}
				\item \hypertarget{item:ia}{} \textbf{IA}: The model structure admits a notion of cell monad, and for every cell monad with parameters, its category of algebras has weakly stable typal \textbf{I}nitial \textbf{A}lgebras with representable lifts.\footnote{See \cref{def:ia} for a precise definition.}
				\item \hypertarget{item:u}{} \textbf{U}: The model structure has sufficient strict univalent \textbf{U}niverses.\footnote{See \cref{def:u} for a precise definition.}
	\end{itemize}
\end{notation}

Next we analyze which properties are preserved by the filter quotient construction.

\begin{theorem} \label{thm:prop}
	Let $\cM$	be a model category and $\Phi$ a model filter on $\cM$.
	\begin{itemize}[leftmargin=*]
		\item The functor $P_\Phi\colon	\cM \to \cM_\Phi$ preserves the following model categorical properties: \hyperlink{item:flc}{\textbf{FLC}}, \hyperlink{item:et}{\textbf{ET}}, \hyperlink{item:lc}{\textbf{LC}}, \hyperlink{item:slc}{\textbf{SLC}}, \hyperlink{item:rp}{\textbf{RP}}, \hyperlink{item:cim}{\textbf{CIM}},	\hyperlink{item:cem}{\textbf{CEM}}, \hyperlink{item:fe}{\textbf{FE}}, \hyperlink{item:tcp}{\textbf{TCP}}, and \hyperlink{item:cl}{\textbf{CL}}.
		\item If $\Phi$ is also a simplicial model filter, then $P_\Phi$ also preserves \hyperlink{item:s}{\textbf{S}}, \hyperlink{item:ia}{\textbf{IA}}, and \hyperlink{item:u}{\textbf{U}}.
	\end{itemize}
\end{theorem}
	
Finally, we observe that the preservation of the model categorical properties in \cref{thm:prop} is sufficient to model the type theoretical constructors.

\begin{theorem} \label{thm:main}
	Let $\cM$ be a model category and $\Phi$ a simplicial model filter on $\cM$. In the following table for every chosen row we have the following statement. 
	\begin{enumerate}[leftmargin=*]
  \item If $\cM$ is a model category from the first column, defined via the properties stated in that row, then $\cM_\Phi$ is a model category from the second column, defined via the properties stated in the second column of that row. 
		\item $\cM,\cM_\Phi$ model all type constructors in the third column that are in the same row and all rows above it.
		\item $P_\Phi\colon \cM \to \cM_\Phi$ preserves all model categorical properties in $\cM$ and all type theoretical constructors.
	\end{enumerate}
	\[\hspace*{-1cm}
	\emph{
		\begin{tabular}{|l|l|l|}
							\hline 
							\textbf{Original Model Structure} & \textbf{Filter Quotient Model Structure} & \textbf{Type Theoretical Constructors} \\ \hline
									Logical model category \cite{arndtkapulkin2011modelstypetheory} & Logical model category  & {\small Unit Type, }$\Sigma${\small-types, }$\Pi${\small-types} \\ 
									\hyperlink{item:flc}{\textbf{FLC}}, \hyperlink{item:tcp}{\textbf{TCP}}, \hyperlink{item:fe}{\textbf{FE}} & \hyperlink{item:flc}{\textbf{FLC}},  \hyperlink{item:tcp}{\textbf{TCP}}, \hyperlink{item:fe}{\textbf{FE}} & \\ \hline
									Type-theoretic model category \cite{shulman2015homotopycanonicity} & Finitary type-theoretic model category & {\small identity types, function extensionality} \\ 
									\hyperlink{item:lc}{\textbf{LC}}, \hyperlink{item:cl}{\textbf{CL}}, \hyperlink{item:rp}{\textbf{RP}}, \hyperlink{item:fe}{\textbf{FE}} & \hyperlink{item:flc}{\textbf{FLC}}, \hyperlink{item:cl}{\textbf{CL}}, \hyperlink{item:rp}{\textbf{RP}}, \hyperlink{item:fe}{\textbf{FE}} & \\ \hline
									Good model category \cite{lumsdaineshulman2020goodexcellent} & Finitary good model category& {\small empty type, Boolean type, coproduct types, } \\ 
									\hyperlink{item:lc}{\textbf{LC}}, \hyperlink{item:s}{\textbf{S}}, \hyperlink{item:cim}{\textbf{CIM}}, \hyperlink{item:cl}{\textbf{CL}}, \hyperlink{item:rp}{\textbf{RP}}, \hyperlink{item:slc}{\textbf{SLC}} & \hyperlink{item:flc}{\textbf{FLC}}, \hyperlink{item:s}{\textbf{S}}, \hyperlink{item:cim}{\textbf{CIM}}, \hyperlink{item:cl}{\textbf{CL}}, \hyperlink{item:rp}{\textbf{RP}}, \hyperlink{item:slc}{\textbf{SLC}}&  {\small pushout types,``cell complex'' types: } \\ 
									& & {\small including spheres and tori} \\ \hline 
									Excellent model category \cite{lumsdaineshulman2020goodexcellent}& Initial algebra closed model category & {\small natural numbers type, W-types, propositional} \\ 
									\hyperlink{item:lc}{\textbf{LC}}, \hyperlink{item:s}{\textbf{S}}, \hyperlink{item:cim}{\textbf{CIM}}, \hyperlink{item:cl}{\textbf{CL}}, \hyperlink{item:rp}{\textbf{RP}}, \hyperlink{item:slc}{\textbf{SLC}}, \hyperlink{item:lp}{\textbf{LP}}, \hyperlink{item:cg}{\textbf{CG}} & \hyperlink{item:flc}{\textbf{FLC}}, \hyperlink{item:s}{\textbf{S}}, \hyperlink{item:cim}{\textbf{CIM}}, \hyperlink{item:cl}{\textbf{CL}}, \hyperlink{item:rp}{\textbf{RP}}, \hyperlink{item:slc}{\textbf{SLC}}, \hyperlink{item:ia}{\textbf{IA}} &  {\small  truncations, James constructions, localizations} \\ \hline 
									Type-theoretic model topos \cite{shulman2019inftytoposunivalent} & Elementary type-theoretic model topos & {\small arbitrarily large univalent universes closed}  \\ 
									\hyperlink{item:lc}{\textbf{LC}}, \hyperlink{item:gt}{\textbf{GT}}, \hyperlink{item:rp}{\textbf{RP}}, \hyperlink{item:s}{\textbf{S}}, \hyperlink{item:cem}{\textbf{CEM}}, \hyperlink{item:lp}{\textbf{LP}}, \hyperlink{item:cg}{\textbf{CG}}, \hyperlink{item:slc}{\textbf{SLC}}, \hyperlink{item:f}{\textbf{F}} & \hyperlink{item:flc}{\textbf{FLC}}, \hyperlink{item:et}{\textbf{ET}}, \hyperlink{item:rp}{\textbf{RP}}, \hyperlink{item:s}{\textbf{S}}, \hyperlink{item:cem}{\textbf{CEM}}, \hyperlink{item:slc}{\textbf{SLC}}, \hyperlink{item:ia}{\textbf{IA}}, \hyperlink{item:u}{\textbf{U}}  & {\small under }$\Sigma${\small- and }$\Pi${\small-types, identity types, binary}  \\  
										& & {\small sum types and containing ``cell complex'' types} \\ \hline 
						\end{tabular}
	}
	\]
\end{theorem}

\begin{remark} \label{rem:difficulty}
	One interesting aspect of \cref{thm:main} is that as we go further down the rows of the table, $P_\Phi$ preserves less and less properties. Concretely, while in the first row $P_\Phi$ preserves everything, and in the next two rows we only lose infinite (co)limits, in the last two rows we see significant differences. Hence, the bulk of \cref{sec:technicalities} focuses on proving that $P_\Phi$ still preserves the necessary type theoretical constructors in the last two rows, even though there are significant model categorical differences.
\end{remark}

Let us focus on one particular instance of the main result. One explicit example of a filter quotient is the \emph{filter product}. Building on the discussions in \cref{ex:filter product} we immediately get the following result.

\begin{corollary}\label{cor:filter product}
	Let $\cM$ be a model category with strict initial object, $I$ a set and $\Phi$	a filter of subsets on $I$. Then all results	from \cref{thm:main} apply to the filter product	$\prod_\Phi \cM$.
\end{corollary}

Let us give some explicit examples of the previous case.

\begin{example}
 Let $\cM$ be a type-theoretic model topos and $\Phi$ a filter of subsets on $I$. Then the underlying category is a topos, meaning the initial object is indeed strict \cite[Proposition IV.4]{maclanemoerdijk1994topos}. Hence, the filter product $\prod_\Phi \cM$ is a model of Martin-L{\"o}f type theory with all the constructors in the third column of \cref{thm:main}, meaning it models homotopy type theory.
\end{example}

Let us restrict to an example of the example, relying on \cref{ex: filter product n}. 

\begin{example} \label{ex:filter product kan}
Let $\sset^{Kan}$ be the Kan model structure on simplicial sets, and $\cF$ a non-principal filter on $\bN$. 
	\begin{enumerate}[leftmargin=*]
		\item By \cref{thm:main}, The filter product $\prod_\cF \sset$ is a model of Martin-L{\"o}f type theory with all the structures stated in \cref{thm:main}, meaning it models homotopy type theory.
		\item By \cref{ex: filter product n}, the filter product $\prod_\cF \sset$ does not have infinite (co)limits, is not locally presentable and the model structure is not cofibrantly generated.
	\end{enumerate}
\end{example}

\begin{remark}
	In \cite{kapulkinlumsdaine2021kanunivalent}, the authors explicitly characterize the universe as the Kan complex of (small) Kan complexes $U$. \cref{thm:main} implies that the universe in $\prod_\cF \sset$ can similarly be explicitly characterized as $(U)_\bN$.
\end{remark}

Let us mention one more general example	that follows from \cref{thm:main} and \cref{ex: filter product n}.

\begin{example}
	Let $n \geq 0$ and let $\sset^{Kan_n}$ be the $n$-truncated Kan model structure on the category of simplicial sets, meaning an object $K$ is fibrant if it is Kan fibrant and $n$-truncated. Then $\sset^{Kan_n}$ is a left Bousfield localization of the Kan model structure and hence an excellent model category. It is, however, not a type theoretic model topos, as the universe of $n$-truncated objects is not itself $n$-truncated. Hence, for every set $I$ and filter $\Phi$ on $I$, the induced model structure on $\prod_\Phi \sset^{Kan_n}$ models all type theoretical constructors in the third column of \cref{thm:main} except for the last row, meaning it does not model univalent universes.   
\end{example}

We end this section with some interesting implications, in particular of \cref{ex:filter product kan}. The binary coproduct type in any model of a type theory will correspond to the universal property of a binary coproduct in category theory. However, we would not expect the universal property of an infinite coproduct to correspond to any type constructor in homotopy type theory. While the previously developed models could not match this intuition, the lack of even countable coproducts in \cref{ex:filter product kan} finally confirms this.

\begin{corollary} \label{cor:independent infinite colimit}
	The existence of infinite (co)limits is independent of homotopy type theory. In particular, the universal property of infinite (co)limits cannot be articulated in the internal language of homotopy type theory.
\end{corollary}

Every model of homotopy type theory has a natural number type, which in our models corresponds to the natural number object. Beyond its universal property, a natural number object in a category can also have the property of being \emph{standard}, meaning the collection of maps $\{n\}\colon1 \to \bN$ is jointly surjective, where $n$ ranges over the external set of natural numbers. We would not expect this property to appear on the type theory side, as even its articulation requires an external perspective. Yet, it is true in every type-theoretic model topos (where the maps $\{n\}$ in fact give us a colimit cocone). Again,  \cref{ex:filter product kan} helps address this, as the natural number object therein is not standard (as proven in \cite[Example D.5.1.7]{johnstone2002elephantii}).

\begin{corollary} \label{cor:independent standard nno}
 Models of homotopy type theory can have standard or non-standard natural number objects.
\end{corollary}

Type theories are by definition \emph{internally well-pointed}, meaning the function type $1 \to X$ is just $X$, and this property does translate to every model, where it manifests as the fact that the internal mapping object $X^1$ is just $X$. However, for models we can also consider \emph{external well-pointedness}. We say a simplicial model category $\cM$ with cofibrant terminal object $1$ is \emph{externally well-pointed} if the mapping space functor $\Map_{\cM}(1,-)\colon\cM \to\sset$ is faithful. In the non-homotopical setting this external notion of well-pointedness is in fact one of the defining properties of Lawvere's \emph{elementary theory of the category of sets (ETCS)} \cite{lawvere1964elementarysets}, see also \cite[Section VI.10]{maclanemoerdijk1994topos}. Up until now, this property was very restrictive for $\infty$-categorical models of HoTT, as the following result demonstrates.

\begin{proposition} \label{prop:wellpointed unique}
	Up to equivalence, there is a unique well-pointed type-theoretic model topos $\cM$.
\end{proposition}

\begin{proof}
 For existence, we simply let $\cM = \sset^{Kan}$. Now, let $\cM$ be a type-theoretic model topos and assume that $\Map_{\cM}(1,-)\colon \cM \to \sset$ is faithful. We denote by $\Map_{\Ho_\infty\cM}(1,-)\colon \Ho_\infty\cM \to \s$ the induced functor on underlying $\infty$-categories, where we are following the common convention $\s  = \Ho_\infty(\sset^{Kan})$. Notice $\Ho_\infty\cM$ is a cocomplete $\infty$-category \cite[Corollary 4.2.4.8]{lurie2009htt}. Hence, it suffices to prove $\Map_{\Ho_\infty\cM}(1,-)$ is in fact full. Indeed, in that case $\Ho_\infty\cM$ is a cocomplete sub-$\infty$-category of $\s$, which, by the universal property of $\s$, means it must be equivalent. 

 By \cite[Proposition A.3.7.6]{lurie2009htt}, $\Ho_\infty\cM$ is a presentable $\infty$-category, meaning, by \cite[Corollary 5.5.2.9]{lurie2009htt}, $\Map_{\Ho_\infty\cM}(1,-)$ has a left adjoint denoted $1 \otimes -\colon \s \to \Ho_\infty\cM$. Now, faithfulness implies that for every object $X$ in $\Ho_\infty\cM$, the counit of the adjunction $1 \otimes \Map_{\Ho_\infty\cM}(1,X) \to X$ is $(-1)$-truncated in $\Ho_\infty\cM$, meaning it is a subobject of $X$. Let $c\colon X \to \Omega$  be the map to the subobject classifier, classifying this subobject. Now, by definition, every map $ 1 \to X$ factors through $1 \otimes \Map_{\Ho_\infty\cM}(1,X)$, which means that $\Map_{\Ho_\infty\cM}(1,c) \simeq \Map_{\Ho_\infty\cM}(1,t)$, where $t\colon X \to \Omega$ classifies $\id_X$. By the well-pointedness assumption, this implies that $c$ and $t$ are equivalent subobjects of $X$, meaning $c$ is a weak equivalence and so $\Map_{\Ho_\infty\cM}(1,-)$ is fully faithful. 
\end{proof}

This result implies that there is a unique model of homotopy type theory that has small colimits and is well-pointed, namely $\sset^{Kan}$. It was already known that dropping the well-pointedness assumption results in many new models (every type theoretic model-topos), however, as of now, it was an open problem whether there are further well-pointed models, even though it would have been expected. With our results at hand, a slight variation of \cref{ex:filter product kan} helps address this matter. 

\begin{example}[Non-Standard Model of Homotopy Type Theory] \label{ex:nonstandard model}
	Let $\cU$ be a non-principal ultrafilter on $\bN$, meaning it is a maximal non-trivial filter in $P\bN$. Then $\prod_\cU \sset$ is a model of Martin-L{\"o}f type theory with all the structures stated in \cref{thm:main}. Moreover, $\prod_\cU \sset$ is externally well-pointed, as $1$ has no non-trivial subobjects, see also \cite[Example 7.2]{rasekh2025filtermodelcat} for a more detailed discussion. However, $\prod_\cU \sset$ is not equivalent to $\sset^{Kan}$, as it does not have infinite coproducts, by \cref{ex: filter product n}.
\end{example}

Filter quotient constructions with respect to ultra filters have been used throughout the literature to construct non-standard models of set theory, in the sense of Lawvere \cite{lawvere1964elementarysets}, as well as models for non-standard analysis \cite{palmgren1997nonstandardanalysis}, see \cite[Example 9.45]{johnstone1977topos} for further details. From this perspective \cref{ex:nonstandard model} can justifiably be considered a \emph{non-standard model for spaces} or \emph{non-standard model of homotopy type theory}.

\begin{remark} \label{rem:infinity}
	At some level the implications of the main theorem, meaning \cref{cor:independent infinite colimit,cor:independent standard nno} and the existence of the model in \cref{ex:nonstandard model}, should be unsurprising to experts and consistent with previous expectations. However, these results do provide a more rigorous understanding of the concept of ``infinity'' in homotopy type theory, as we shall explain.
 
	On the one side, homotopy type theory inherently carries an internal notion of infinity, given via its natural numbers and various constructions built on top of it. On the other side, any notion of $\infty$-category is inherently infinite (it involves an infinite layer of morphisms and coherences). Hence the second we aim to construct a model of homotopy type theory, we secretly have two notions of infinity that we are working with, one coming from homotopy type theory, and one from our notion of $\infty$-category that we have chosen to construct our models. In many cases these do coincide. Those cases include the locally presentable models, where the existence of small colimits guarantees that any internal notion of ``infinity'', such as the natural number object, matches with external notions of ``infinity'', in this case meaning the natural number object is given as an infinite coproduct.

	However, in this paper, we have constructed new models where these two notions of infinity visibly diverge, resulting in the kind of results stated above. While this divergence might not always matter, it will play a crucial role when trying to analyze inherently infinite notions. A simple example occurs when we try to compare $\infty$-categories internal to homotopy type theory with $\infty$-categories internal to models, at which point the diverging infinities result in a discrepancy. See \cite{rasekh2025shott} for further details.
\end{remark}

\section{Technicalities and Proofs} \label{sec:technicalities}
This section is dedicated to technical definitions and detailed proofs. As explained in \cref{rem:difficulty}, most model categorical properties straightforwardly transfer from a model category to its filter quotient. However, two properties require more care and are considered separately: the construction of higher inductive types (\cref{subsec:algebras}) and the construction of univalent universes (\cref{subsec:universes}). Having covered the two challenging situations, we can finally complete the proof in \cref{subsec:proof}. 

\subsection{Initial Algebras} \label{subsec:algebras}
In this subsection we prove that filter quotient model categories inherit the construction of higher inductive types via the existence of initial algebras from the original model category. Before we proceed with the proof, we review the original argument by Lumsdaine and Shulman and why it needs to be adjusted \cite{lumsdaineshulman2020goodexcellent}.

Intuitively, for a given monad $\bT\colon \C \to \C$, a higher inductive type is a suitably initial $\bT$-algebra object in $\C$ that is weakly stable under pullbacks. To make this argument precise, the authors proceed with the following three steps.

\begin{enumerate}[leftmargin=*]
	\item First they provide a precise definition of initiality, via weakly stable typal initial $\bT$-algebras with representable lifts \cite[Definition 12.4, Definition 12.5]{lumsdaineshulman2020goodexcellent}, and observe it has the desired semantic implications, giving us higher inductive types \cite[Theorem 12.8]{lumsdaineshulman2020goodexcellent}. 
	\item They then define a notion of \emph{cell monad with parameters} \cite[Definition 12.9]{lumsdaineshulman2020goodexcellent} in an excellent model category. It is point-wise defined as a (possibly infinite) composition of pushouts along monad cells, which are, in turn, monads freely generated by polynomial endofunctors  \cite[Definition 11.10]{lumsdaineshulman2020goodexcellent}. 
	\item Finally, the authors prove that if $\cM$ is an excellent model category, and $\bT$ is a cell monad with parameters, then $\cM$ has weakly stable typal initial $\bT{-}\Algf$-algebras and representable lifts \cite[Theorem 12.13, Theorem 12.14]{lumsdaineshulman2020goodexcellent}.
\end{enumerate}

Unfortunately, in two of these three steps they use the fact that excellent model categories are locally presentable and cofibrantly generated. Indeed, a monad cell is defined as a free monad generated via polynomial functors, which requires local presentability \cite[Lemma 11.9]{lumsdaineshulman2020goodexcellent}, cell monads with parameters are defined as transfinite compositions, which also need infinite colimits \cite[Definition 11.10]{lumsdaineshulman2020goodexcellent}, and the construction of the typal initial $\bT{-}\Algf$-algebra requires the small object argument \cite[Theorem 11.13]{lumsdaineshulman2020goodexcellent}. As local presentability and cofibrant generation are not preserved by the filter quotient construction, we cannot use these assumptions anymore. Instead, we need to isolate the relevant property that will be preserved by the filter quotient construction. This motivates focusing on \emph{fibred categories of structures with parameters (FCoSwP)} \cite[Definition 12.3]{lumsdaineshulman2020goodexcellent}, which is an abstraction employed by the authors to bundle the data of a category of fibrant objects, its category of fibrations and a category of algebras of a monad \cite[Lemma 12.11]{lumsdaineshulman2020goodexcellent}. We hence proceed as follows:

\begin{enumerate}[leftmargin=*]
	\item We prove that a fibred category of structures with parameters on a category induces one on its filter quotient (\cref{prop:fcoswp filter quotient}). 
	\item Moreover, we show that the projection preserves weakly stable typal initial algebras or representable lifts (\cref{thm:filter fcoswp}).
	\item Finally, we define cell monads with parameters on a filter quotient category (\cref{def:cell monad wp}), and prove its category of algebras has weakly stable typal initial algebras and representable lifts (\cref{thm:ia})
\end{enumerate}

These results suggests that instead of excellent model categories, we need good model categories that satisfy the following property \hyperlink{item:ia}{\textbf{IA}}. 

\begin{definition}[Precise formulation of \hyperlink{item:ia}{\textbf{IA}}] \label{def:ia}
	A model category $\cM$ satisfies \hyperlink{item:ia}{\textbf{IA}} if there is a class of fibred monads, the \emph{cell monads}, such that for every cell monad with parameters the associated FCoSwP of $\bT$-algebras has weakly stable typal initial $\bT$-algebras with representable lifts.
\end{definition}

From this perspective the main result can be summarized as stating that the filter quotient construction preserves \hyperlink{item:ia}{\textbf{IA}}. See \cref{thm:ia} for a precise statement. The axiomatic treatment of cell monads in this context might be unexpected, however, we do observe that if we start with the example of cell monads in an excellent model category, then its induced notion of cell monad on filter quotient does not in fact lose any example of interest, and in particular every building block, the monad cells (\cref{prop:free polynomial monad}). Having better understood the required conditions and anticipated results, we now dive into the precise definitions and results.

Let $(\C,\cT)$ be a comprehension category, meaning $\cT$ is a Grothendieck fibration over $\C$ with a Cartesian functor to the target projection $\C^\rightarrow \to \C$. Recall the following notion introduced in \cite[Definition 12.1]{lumsdaineshulman2020goodexcellent}. A \emph{parameter scheme} over $(\C,\cT)$ is a finite set $\cP$ whose elements are either type parameters or term parameters. To every parameter scheme, we associate a fibration, $\Inst(\cP) \to \C$, called the \emph{instantiation of $\cP$} defined inductively via the elements of $\cP$. See \cite[Definition 12.1]{lumsdaineshulman2020goodexcellent} for the inductive characterization.

\begin{definition}
	Let $\C$ be a category. A \emph{fibred category of structures with parameters (FCoSwP)} over $\C$ consists of the following data.
	\begin{enumerate}[leftmargin=*]
		\item A comprehension category $(\C,\cT)$
		\item A finite set $\cP$, called the \emph{parameter scheme}, with elements either \emph{type parameters} or \emph{term parameters}.
		\item The associated \emph{instantiation of $\cP$}, $\Inst(\cP) \to \C$, which is the Grothendieck fibration defined via $\cP$.
		\item A Grothendieck fibration $\cS \to \C$ along with a faithful isofibration $\cS \to \Inst(\cP) \times_{\C} \C^{\rightarrow}$ that is amnestic, meaning the only isomorphism that is mapped to the identity is itself the identity.
	\end{enumerate}
\end{definition}

\begin{notation} \label{not:fcoswp}
	As part of this definition, we adopt the following notational conventions and terminology:
	\begin{itemize}[leftmargin=*]
		\item For a given object $\Gamma$ in $\C$, we use $H$ in $\cT(\Gamma)$ to denote an object in the category $\cT$ over $\Gamma$. We use similar notations for $\cS$ and $\Inst(\cP)$.
		\item For an object $H$ in $\cT(\Gamma)$, we denote its image in $\C^\rightarrow$ by $\Gamma.H \to \Gamma$.
		\item For an object $(\theta,X \to \Gamma)$ in $\Inst(\cP) \times_{\C} \C^{\rightarrow}$ (meaning $\theta$ is in $\Inst(\cP)(\Gamma)$) we denote an object in the fiber an $\cS(\theta)$-structure on $X$.
	\end{itemize}
\end{notation}

We now want to show that an FCoSwP over $\C$ induces an FCoSwP over $\C_\Phi$. This requires several definitions and lemmas. First we have the following simple lemma.

\begin{lemma} \label{lemma:pullback	comprehension}
	Let $(\C,\cT)$ be a comprehension category and $F\colon\F \to \C$ a discrete fibration. Then $(\F,F^*\cT)$ is a comprehension category over $\F$ and the functor $\F^*\cT \to \cT$ preserves and reflects Cartesian morphisms.
\end{lemma}

\begin{proof}
 As $F$ is a discrete fibration, the pullback of $\C^{\rightarrow} \to \C$ along $F$ is precisely $\F^{\rightarrow} \to \F$. Moreover, Grothendieck fibrations and Cartesian functors are evidently pullback stable. Finally, it is a direct observation that a morphism $(f,g)$ in $F^*\cT$ is Cartesian if and only if $f$ is Cartesian in $\cT$ \cite[Proposition 8.1.15]{borceux1994handbook2}.
\end{proof}

For the next definition, recall that for a comprehension category $(\C,\cT)$, the categories $\D_1(\cT)$ ($\D(\cT)$) have objects dependent projections (display maps), and both have morphisms given by pullback squares. Finally, we denote by $\cT_{\cong} \to \C$ the (faithful) subcategory of $\cT$ with the same objects and only Cartesian morphisms. For more details see \cite[Section 12]{lumsdaineshulman2020goodexcellent}. We need the following observation regarding these definitions.

\begin{lemma} \label{lemma:pullback ddt}
 Let $(\C,\cT)$ be a comprehension category, $F\colon\F \to \C$ a discrete fibration. Then we have $F^*\D(\cT) \cong \D(F^*\cT)$, $F^* \D_1(\cT) \cong \D_1(F^*\cT)$, $F^* \D_{1,*}(\cT) \cong \D_{1,*}(F^*\cT)$ and $F^* \cT_{\cong} \cong (F^*\cT)_{\cong}$.
\end{lemma}

\begin{proof}
 \cref{lemma:pullback comprehension} implies that the pullback of dependent projections (display maps) along the discrete fibrations $F$ are again dependent projection (display maps). This implies the first three isomorphisms. The last part follows from the fact that the pullback preserves and reflects Cartesian morphisms (\cref{lemma:pullback comprehension}).
\end{proof}

\begin{lemma} \label{lemma:pullback instantiation}
	Let $(\cC,\cT)$ be a comprehension category, $\cP$ a parameter scheme and $\Inst(\cP) \to \C$ the associated instantiation. Let $F\colon\F \to \C$ be a discrete fibration. Then there is a pullback parameter scheme $F^*\cP$, for which the associated instantiation is precisely $F^*\Inst(\cP) \to \F$.
\end{lemma} 

\begin{proof}
 We use the mutually inductive definition in \cite[Definition 12.1]{lumsdaineshulman2020goodexcellent} to prove the result similarly via mutual induction.
	\begin{itemize}[leftmargin=*]
		\item If $\cP$ is empty, then we take $F^*\cP$ to be empty as well, and we indeed have $F^*\C = \F$.
		\item Let us assume that  $\cP$ is extended by a type parameter $\alpha\colon\Inst(\cP) \to \D(\cT)$. Then $\Inst(\langle F^*\cP,F^*\alpha\rangle) \to \Inst(\langle F^*\cP\rangle)$ is by definition the pullback of $F^*\cT_{\cong} \to \F$ and hence by pullback gluing, following \cref{lemma:pullback ddt}, also the pullback of $\cT_{\cong} \to \C$. Similarly, $\Inst(\langle\cP,\alpha\rangle) \to \Inst(\cP)$ is the pullback of $\cT_{\cong} \to \C$, by definition. Hence, by pullback cancellation we get the desired isomorphism $F^*\Inst(\langle\cP,\alpha\rangle) \cong \Inst(\langle F^*\cP,F^*\alpha\rangle)$. This means we extend $F^*\cP$ by the type parameter $F^*\alpha\colon F^*\Inst(\cP) \to \D(F^*\cT)$, using the fact that $\D(F^*\cT) \cong F^*\D(\cT)$ (\cref{lemma:pullback ddt}), finishing this induction step.
		\item Let us assume that $\cP$ is extended by a term parameter $\alpha\colon\Inst(\cP) \to \D(\cT)$, $\beta\colon \Inst(\cP) \to \D_1(\cT)$. Then, employing an analogous pullback cancellation argument as in the previous step, it follows that $F^*\Inst(\llbracket\cP,\alpha,\beta\rrbracket) \cong \Inst(\llbracket F^*\cP,F^*\alpha,F^*\beta\rrbracket)$. This means we extend $F^*\cP$ by the term parameter $F^*\alpha\colon F^*\Inst(\cP) \to \D(F^*\cT)$ and $F^*\beta\colon F^*\Inst(\cP) \to \D_1(F^*\cT)$, finishing the induction step.\qedhere
	\end{itemize}
\end{proof}	

\begin{lemma} \label{lemma:pullback fcoswp}
 Let $F\colon\F \to \C$ be a discrete fibration, and $(\C,\cT,\cS,\cP, \Inst(\cP))$ an FCoSwP over $\C$. Then \[(\F,F^*\cT,F^*\cS, F^*\cP, \Inst(F^*\cP))\] is an FCoSwP over $\F$.
\end{lemma}

\begin{proof}
	First of all, Grothendieck fibrations, isofibrations, faithful functors and amnestic functors are evidently pullback stable. Moreover, by \cref{lemma:pullback comprehension}, $F^*\cT$ is still a comprehension category. Finally, by \cref{lemma:pullback instantiation}, $F^*\cP$ is a parameter scheme with instantiation $\Inst(F^*\cP)$. Hence we are done.
\end{proof}

Recall that for an object $C$ in $\C$, the projection $\pi_D\colon\C_{/D} \to \C$ is a discrete fibration, motivating the following definition. 

\begin{definition} \label{def:pullback fcoswp}
	Let $(\C,\cT,\cS,\cP, \Inst(\cP))$ be an FCoSwP over $\C$ and $C$ an object in $\C$. Let $(\C_C,\cT_C,\cS_C,\cP_C, \Inst(\cP_C))$ denote the FCoSwP obtained by applying \cref{lemma:pullback fcoswp} to $\pi_C\colon\C_{/C} \to \C$. 
\end{definition}

We record here the following basic fact about Grothendieck fibrations. 

\begin{lemma} \label{lemma:fibration terminal}
 Let $F\colon\F \to \C$ be a Grothendieck fibration. If $\F$ has a terminal object $\hat{1}$, then $F(\hat{1})$ in $\C$ is subterminal. Moreover, if an object in $\C$ does not admit a morphism to $F(\hat{1})$, meaning it is not in the full subcategory $\C_{F(\hat{1})}$, then its fiber is empty.
\end{lemma}

\begin{remark} \label{rem:terminal}
	\cref{lemma:fibration terminal} implies that for a Grothendieck fibration $F\colon \F \to \C$ with terminal object $\hat{1}$, we can without loss of generality restrict the codomain to $\C_{F(\hat{1})}$, at which point the image of $\hat{1}$ will in fact be terminal. We will henceforth assume that if $\F$ has a terminal object, its image is indeed terminal in $\C$.
\end{remark}

\begin{notation}
 For a given category $\C$ with finite products and subterminal object $U$, we denote the functor $- \times U\colon \C \to \C_U$ by $P_U\colon \C \to \C_U$. 	
\end{notation}

\begin{lemma} \label{lemma:adj fcoswp}
	Let $\C$ be a category, $F\colon\F \to \C$ a Grothendieck fibration, and $U$ a subterminal object in $\C$. 
	\begin{enumerate}[leftmargin=*]
	 \item The fully faithful functor $\F_U = \C_U \times_\C \F \to \F$ admits a Cartesian right adjoint that makes the following diagram of adjunctions commute 
	 \[
	 \begin{tikzcd}[column sep= 2cm ]
		 \F_{U} \arrow[d, "F_U"]  \arrow[r, bend left=15, "\pi"] \arrow[r, leftarrow, bend right=15, "\bot", dashed]& \F \arrow[d, "F"] \\ 
		 \C_{U} \arrow[r, bend left=15, "\pi"] \arrow[r, leftarrow, bend right=15, "\bot", "P_U"'] & \C
	 \end{tikzcd}.
	 \]
		\item Assume that $\F$ has a terminal object $\hat{1}$ (recall \cref{rem:terminal}). Let $\hat{U}$ denote the domain of the Cartesian lift of $U \to 1$ along $\hat{1}$. Then $\hat{U}$ is subterminal and the right adjoint to $\F_U$ is given by $P_{\hat{U}}$.
	\end{enumerate}
\end{lemma}

\begin{proof}
	$(1)$ First we specify the right adjoint on objects. For a given object $X$ in $\F$, let $\alpha_X\colon X_U \to X$ denote the Cartesian lift of $\pi\colon F(X) \times U \to F(X)$ along $X$. By construction, $X_U$ is an object in $\F_U$, as $F(X_U) = X \times U$ lands in $\C_U$. This is the value of the functor on $X$. Next, we specify the right adjoint on morphisms. For a given morphism $f\colon X \to Y$ in $\F$, using the fact that $Y_U \to Y$ is Cartesian, the following diagram admits a factorization
	\begin{equation} \label{eq:cart lift}
		\begin{tikzcd}
			X_U \arrow[r, "f_U", dashed] \arrow[d, "\alpha_X"] & Y_U \arrow[d, "\alpha_Y"'] \\ 
			X \arrow[r, "f"] & Y
		\end{tikzcd}
	\end{equation} 
	making the diagram commute. This is the value of the functor for a morphism $f$. Functoriality immediately follows from the fact that Cartesian lifts are unique.	

	We now show that this functor is indeed the right adjoint to the inclusion $\F_U \to \F$. It suffices to observe that every arbitrary morphism $g\colon Z \to X$ in $\F$, with $Z$ in $\F_U$, uniquely factors through $X_U \to X$. By assumption $F(Z)$ admits a map to $U$, meaning there is a factorization $F(Z) \to F(X) \times U \to F(X)$ of $F(g)$. Hence, using the fact that $\alpha_X\colon X_U \to X$ is the Cartesian lift of $F(X) \times U \to F(X)$ along $X$, there is indeed the unique factorization $Z \to X_U \to X$ of $g$. Finally, if $f$ is a Cartesian morphism, then in \ref{eq:cart lift} the bottom and vertical morphisms are Cartesian, which implies that the top map $f_U$ is also Cartesian, which proves the right adjoint preserves Cartesian morphisms.

	$(2)$ First we prove that $\hat{U}$ is subterminal. Let $f,g \colon Z \to \hat{U}$ be two morphisms in $\F$. Then $F(f) = F(g)$ as both have codomain the subterminal object $U$, meaning both are lifts of $F(Z) \to U \to 1$, along the Cartesian morphism $\hat{U} \to \hat{1}$ and hence, by uniqueness, need to be equal. We now want to observe that $P_{\hat{U}}$ is indeed the right adjoint to the inclusion $\F_U \to \F$. This will be a formal implication of the fact that $\hat{U}$ is terminal in $\F_U$, which follows directly. Indeed, for an arbitrary object $Z$ in $\F_U$, by definition there is a factorization $F(Z) \to Y \to 1$, which means the unique map $Z \to \hat{1}$ factors through the Cartesian morphism $\hat{U} \to \hat{1}$, giving us $Z \to \hat{U} \to \hat{1}$.
\end{proof}

Before we proceed we recall the following technical result regarding filter quotients, that plays an important role throughout this section. For a proof see \cite[Theorem V.9.2]{maclanemoerdijk1994topos}.

\begin{lemma} \label{lemma:filtered colimits filter quotients}
	Let $\C$ be a category with with finite products and $\Phi$ be a filter of subterminal objects. Then $\C_\Phi$ is the filtered	colimit of the diagram $\C_{(-)}\colon\Phi^{op} \to \cat$, where functoriality is given by mapping $V \leq U$, to $P_V\colon \C_U \to \C_V$.
\end{lemma}

In this section we will primarily use this description of the filter quotient as a filtered colimit. From this perspective an object in $\C_\Phi$ is of the form $U \times X$, where $U$ is in $\Phi$ and $X$ in $\C$, and a morphism from $X \to Y$, is a morphism of the form $f\colon U \times X \to U \times Y$, where objects and morphisms are identified in a manner similar to \cref{def:filter quotient category}. We now use this description as a filtered colimit, to show that all properties of FCoSwP transfer to the filter quotient. First we record the following basic result regarding Grothendieck fibrations.

\begin{lemma}\label{lemma:filtered colimit fibration}
	The filtered colimit of Grothendieck fibrations and Cartesian functors is a Grothendieck fibration. Similarly, the filtered colimit of isofibrations is an isofibration.
\end{lemma}

\begin{lemma} \label{lemma:filter quotient fibration}
	Let $F\colon\F \to \C$ be a Grothendieck fibration, and let $\Phi$ be a filter of subterminal objects on $\C$.
	\begin{enumerate}[leftmargin=*]
		\item The filter $\Phi$ induces a diagram of Grothendieck fibrations $\F_U \to \C_U$ and Cartesian functors, compatible with the functors $P_V\colon \C_U \to \C_V$. 
		\item The resulting colimit of the Grothendieck fibrations $F_U$,  
		\[\underset{U \in \Phi^{op}}{\colim} F_U = F_\Phi\colon \F_{\Phi} \to \C_\Phi,\] 
		is a Grothendieck fibration. 
		\item If $\F$ has a terminal object $\hat{1}$, then the $\F_\Phi$ is itself the filter quotient of $\F$.  
	\end{enumerate}
\end{lemma}
	
\begin{proof}
	$(1)$ For a given $V \leq U$ in $\Phi$, by \cref{lemma:adj fcoswp}, the functor $P_V\colon \C_U \to \C_V$ lifts to a commutative diagram of Cartesian functors $\F_U \to \F_V$, giving us the desired natural diagram $F_U\colon \F_U \to \C_U$. 

	$(2)$ By the previous step, $F_{(-)}$ is a natural transformation between the diagrams $\F_{(-)}, \C_{(-)}\colon \Phi^{op} \to \cat$. This induces a functor on colimits $\F_\Phi \to \C_\Phi$. By \cref{lemma:filtered colimit fibration}, this functor is a Grothendieck fibration.

	$(3)$ If $\F$ has a terminal object $\hat{1}$, then by \cref{lemma:adj fcoswp}, the filter $\Phi$ lifts to a filter on $\F$, also denoted $\Phi$, and, by \cref{lemma:filtered colimits filter quotients}, the induced filter quotient coincides with the filtered colimit $\F_\Phi$. 
\end{proof}

The following useful lemma follows directly from the bijection invariance of colimits and \cref{lemma:pullback ddt}.

\begin{lemma} \label{lemma:filter quotient ddt}
 Let $(\C,\cT)$ be a comprehension category and $\Phi$ a filter of subterminal objects on $\C$. Then we have $\D(\cT)_\Phi \cong \D(\cT_\Phi)$, $\D_1(\cT)_\Phi \cong \D_1(\cT_\Phi)$, $\D_{1,*}(\cT)_\Phi \cong \D_{1,*}(\cT_\Phi)$ and $(\cT_{\cong})_\Phi \cong (\cT_\Phi)_{\cong}$.
\end{lemma}

\begin{lemma} \label{lemma:instantiation filter quotient}
	Let $(\cC,\cT)$ be a comprehension category, $\cP$ a parameter scheme and $\Inst(\cP) \to \C$ the associated instantiation, and let $\Phi$ be a filter of subterminal objects on $\C$. $\cP_\Phi$, defined as the filtered colimit of $\cP_U$ (\cref{def:pullback fcoswp}), is a parameter scheme over $\C_\Phi$, for which the associated instantiation is precisely given by the filtered colimit $\Inst(\cP)_\Phi \to \C_\Phi$.
\end{lemma}

\begin{proof}
	We can precisely repeat the three inductive steps and the arguments therein from the proof of \cref{lemma:pullback instantiation}, this time relying on the isomorphisms from \cref{lemma:filter quotient ddt}, instead of \cref{lemma:pullback ddt}.
\end{proof}

\begin{proposition} \label{prop:fcoswp filter quotient}
	Let $(\C,\cT,\cS, \cP, \Inst(\cP))$ be an FCoSwP over $\C$ such that $\cT,\cS$ have terminal objects, and $\Phi$ a filter of subterminal objects. Then $(\C_\Phi,\cT_{\Phi}, \cS_{\Phi}, \cP_{\Phi}, \Inst(\cP_\Phi))$ is an FCoSwP over $\C_\Phi$. 
\end{proposition}

\begin{proof}
	First, by \cref{lemma:filter quotient fibration}, the induced functors on filter quotients $\cT_\Phi \to \C_\Phi$, $\cS_\Phi \to \C_\Phi$ are still Grothendieck fibrations. Moreover, by \cref{lemma:filtered colimit fibration}, $\cS_\Phi \to \Inst(\cP_\Phi) \times_{\C_\Phi} \C_\Phi^{\rightarrow}$ is additionally an isofibration. Next, by \cref{lemma:instantiation filter quotient}, $\cP_\Phi$ is a parameter scheme over $\C_\Phi$, with instantiation $\Inst(\cP_\Phi)$. Now, the fact that filtered colimits commute with finite limits implies that the filtered colimits preserve monomorphisms. So, filtered colimits preserve faithfulness, meaning $S_\Phi \to \Inst(\cP_\Phi) \times_{\C_\Phi} \C_\Phi^{\rightarrow}$ is faithful.

	Finally, a morphism $[f]$ in the filter quotient $S_\Phi$ is an isomorphism if there is a $U$ in $\Phi$, such that its representative $f$ in $S_U$ is an isomorphism. Now, if the image of $[f]$ is the identity, then so is the image of $f \times V$, for some $V$. As, by \cref{lemma:pullback fcoswp}, the functor $S_V \to \Inst(\cP_V) \times_{\C_V} (\C_V)^{\rightarrow}$ is amnestic, it follows that the image of $f \times V$ is the identity, which means the image of $[f]$ is the identity. This proves the functor is amnestic, hence we are done.
\end{proof}

Having constructed the filter quotient FCoSwP, we now prove it preserves relevant properties of the original FCoSwP. First, we note here the following basic categorical lemma. 
\begin{lemma} \label{lemma:adjunctions}
 Let $F\colon\C^{op} \to \set$ be a representable functor with representing object $C$. Let $L\colon\D \to \C$ be a functor with right adjoint $R$, then $F \circ L$ is also representable, with representing object $R(C)$. 
\end{lemma}

Finally, we can move on to the main result regarding FCoSwP and filter quotients. Here we refer the reader to \cite[Definition 12.4]{lumsdaineshulman2020goodexcellent} for the definition of weakly stable typal initial $\cS$-algebras and \cite[Definition 12.5]{lumsdaineshulman2020goodexcellent} for the definition of representable lifts.

\begin{lemma} \label{lemma:projections fcoswp}
	Let $(\C,\cT,\cS, \cP, \Inst(\cP))$ be an FCoSwP over $\C$ and $U$ be a subterminal object. 
	\begin{enumerate}[leftmargin=*]
		\item $P_U\colon \C \to \C_U$ maps weakly stable typal initial $\cS$-algebras to weakly stable typal initial $\cS_U$-algebras.
		\item $P_U\colon \C \to \C_U$ preserves representable lifts.  
	\end{enumerate}
\end{lemma}

\begin{proof}
	$(1)$ By definition weakly stable typal initial $\cS$-algebras are stable under pullback, in this case $P_U\Gamma \to \Gamma$. 

	$(2)$ We need to prove the $\cS$-lift functor remains representable, with representing objects $P_U\overline{\omega}$,  after precomposition  with the fully faithful functor $\C_{/P_U\Gamma} \to \C_{/\Gamma}$. This follows directly from \cref{lemma:adjunctions}, as $P_U$ is the right adjoint to the inclusion.  
\end{proof}

\begin{theorem} \label{thm:filter fcoswp}
	Let $\C$ be a category with finite products, dependent exponentials of display maps, and product projections along display maps. Let $(\C,\cT,\cS, \cP, \Inst(\cP))$ be an FCoSwP over $\C$ such that $\cT,\cS$ have terminal objects, and $\Phi$ a filter of subterminal objects.
	\begin{enumerate}[leftmargin=*]
		\item  $(\C_\Phi,\cT_{\Phi}, \cS_{\Phi},\cP, \Inst(\cP)_\Phi)$ has weakly stable typal initial $\cS_\Phi$-algebras if there exists a $U$ in $\Phi$, such that the FCoSwP $(\C_U,\cT_U, \cS_U, \cP_U, \Inst(\cP_U))$ has weakly stable typal initial $\cS_U$-algebras
		\item  $\cS$ has representable lifts if there exists a $U$ in $\Phi$, such that $\cS_U$ has representable lifts. 
		\item If there exists a $U$ in $\Phi$, such that $\C_U$ has weakly stable typal initial $\cS_U$-algebras and $\cS_U$ has representable lifts, then $(\C_\Phi)_!$ has strictly stable typal initial $(\cS_\Phi)_!$-algebras
	\end{enumerate}
\end{theorem}

\begin{proof}
	(1) Without loss of generality we can assume $U$ is the terminal object. For $\Gamma$ in $\C$ and $\theta$ in $\cP$, let $H_\Gamma$ along with an $\cS(\theta)$-structure on $\Gamma.H \to \Gamma$ be a choice of weakly stable typal initial $\cS(\theta)$-algebra, in the sense of \cite[Definition 12.4]{lumsdaineshulman2020goodexcellent}. We want to prove that $H_\Gamma$ along with the induced $\cS_\Phi(\theta)$-structure, obtained by applying $P_\Phi\colon \cS \to \cS_\Phi$, is still a weakly stable typal initial $\cS_\Phi(\theta)$-algebra. Evidently $P_\Phi$ preserves reindexing of morphisms, hence the collection of objects is still weakly stable, meaning we only need to prove $H_\Gamma$ is in fact a typal initial $\cS_\Phi(\theta)$-algebra.

 Let $C$ be an object in $\cT_\Phi(\Gamma.H)$, meaning there exists a $U_1$ in $\Phi$, such that $C$ is in $\cT_{U_1}(P_{U_1}\Gamma.P_{U_1}H)$. Moreover, pick an $\cS_\Phi(\theta)$-structure on the composite $\Gamma.H.C \to \Gamma.H \to \Gamma$, which means there exists a $U_2$ in $\Phi$, such that $P_{U_2}\Gamma.P_{U_2}H.P_{U_2}C \to P_{U_2}\Gamma$, carries a $\cS_{U_2}(P_{U_2}\theta)$-structure. Finally, assume $\Gamma.H.C \to \Gamma.H$ is an $\cS_\Phi(\theta)$-morphism, meaning there exists a $U_3$ in $\Phi$, such that $P_{U_3}\Gamma.P_{U_3}H.P_{U_3}C \to P_{U_3}\Gamma.P_{U_3}H$ is a $\cS_{U_3}(P_{U_3}\theta)$-morphism.
	
	Let $U_0 \leq U_1, U_2,U_3$ in $\Phi$. Then, by \cref{lemma:pullback fcoswp}, $P_{U_0}C$ is in $\cT_{U_0}(P_{U_0}\Gamma.P_{U_0}H)$, and $P_{U_0}\Gamma.P_{U_0}H.P_{U_0}C \to P_{U_0}\Gamma$, carries a $\cS_{U_0}(P_{U_0}\theta)$-structure, such that $P_{U_0}\Gamma.P_{U_0}H.P_{U_0}C \to P_{U_0}\Gamma.P_{U_0}H$ is a $\cS_{U_0}(P_{U_0}\theta)$-morphism. As, by \cref{lemma:projections fcoswp}, $P_{U_0}H_{P_{U_0}\Gamma}$ is still a typal initial $\cS_{U_0}$-algebra, the desired section exists. Applying $P_\Phi\colon \cS_{U_0} \to \cS_\Phi$ implies that the section also exists in the filter quotient, finishing the proof.
	
	(2) Again, without loss of generality we can assume $U$ is the terminal object. Let $\Gamma$ be an object in $\C_\Phi$. Next, let $A$ be in $\cT_\Phi(\Gamma)$, meaning there exists a $U_1$ in $\Phi$, such that $A$ is in $\cT_{U_1}(P_{U_1}\Gamma)$. Let $B$ be in $\cT_\Phi(\Gamma.A)$, meaning there exists a $U_2$ in $\Phi$, such that $B$ is in $\cT_{U_2}(P_{U_2}\Gamma.P_{U_2}A)$. Moreover, let $\theta$ be in $\Inst(\cP_\Phi)$, meaning there exists a $U_3$ in $\Phi$, such that $\theta$ is in $\Inst(\cP_{U_3})$. Next, let $\A$ be a $\cS_\Phi(\theta)$-structure on $\Gamma.A \to \Gamma$, meaning there exists $U_4$ in $\Phi$, such that $\A$ is a $\cS_{U_4}(P_{U_4}\theta)$-structure on $P_{U_4}\Gamma.P_{U_4}A \to P_{U_4}\Gamma$. 
	
	Let $U_0 \leq U_1, U_2, U_3, U_4$ in $\Phi$. Then, by \cref{lemma:pullback fcoswp}, $P_{U_0}\Gamma$ is an object in $\C_{U_0}$, $P_{U_0}A$ is in $\cT_{U_0}(P_{U_0}\Gamma)$, $P_{U_0}B$ is in $\cT_{U_0}(P_{U_0}\Gamma.P_{U_0}A)$, $P_{U_0}\theta$ is in $\Inst(\cP_{U_0})$, and $P_{U_0}\A$ is a $\cS_{U_0}(P_{U_0}\theta)$-structure. Hence, by \cref{lemma:projections fcoswp}, there exist a representable lifts $\overline{\omega} \colon V_{\A,B} \to \Gamma_0$ in $\C_{U_0}$, the equivalence class of which gives us the desired representable lift in $\C_\Phi$ after applying $P_\Phi\colon \C_{U_0} \to \C_\Phi$.

	(3) This is a direct application of the two previous parts and \cite[Theorem 12.8]{lumsdaineshulman2020goodexcellent}.
\end{proof}

Because of \cref{thm:filter fcoswp}, we now have a much better understanding of various properties of $FCoSwP$ over a filter quotient category. We now proceed to study fibred monads and how they interact with filter quotients. This requires reviewing essential aspects of the theory of fibred monads.

Recall that a \emph{fibred monad on $\C$} is a monad $\bT$ on $\C^\rightarrow$ in the category of Grothendieck fibrations over $\C$. See \cite[Definition 11.6]{lumsdaineshulman2020goodexcellent}, for a more explicit description. More generally, a \emph{fibred monad with parameters} is given by a Cartesian functor $\bT\colon \Inst(\cP) \to \Mnd(\C_\Phi)$, where $\Mnd(\C_\Phi) \to \C$ is the Grothendieck fibration with objects over $c$ in $\C$ given by fibred monads on $\C_{/c}$ \cite[Definition 11.6]{lumsdaineshulman2020goodexcellent}. Following \cite[Lemma 12.10]{lumsdaineshulman2020goodexcellent}, every fibred monad with parameters on $\C_\Phi$ induces a FCoSwP on $\C_\Phi$, given by $(\C_\Phi,\Fib,\bT{-}\Algf, \cP_\Phi, \Inst(\cP)_\Phi)$. In particular, for a suitable model category $\cM$ and fibred monad with parameters $\bT$ we have an FCoSwP over the full subcategory of fibrant objects $(\cM_\Phi)_{\textbf{\text{f}}}$, given by $((\cM_\Phi)_{\textbf{\text{f}}},\Fib,\bT{-}\Algf, \cP_\Phi,\Inst(\cP)_\Phi)$. 

We now want a suitable condition on the fibred monad $\bT$, such that the associated FCoSwP has weakly stable typal initial $\bT$-algebras with representable lifts. As we explained in the beginning of this section, we cannot define cell monads in the same way as in \cite{lumsdaineshulman2020goodexcellent}, given the lack of local presentability. However, \cref{thm:filter fcoswp} suggests obtaining fibred monads on $\C_\Phi$ out of fibred monads on $\C$. Making this precise requires several lemmas. The following lemma is a direct implication of the definition of a fibred monad.

\begin{lemma} \label{lemma:pullback mnd}
 Let $\C$ be a category and $U$ a subobject in $\C$. Then the following is a pullback diagram
	\[ 
	\begin{tikzcd}
		 \Mnd(\C_U) \arrow[r] \arrow[d] & \Mnd(\C) \arrow[d] \\
			\C_U \arrow[r] & \C
	\end{tikzcd}
	\]
\end{lemma}

\begin{lemma} \label{lemma:mnd times u}
	 Let $\C$ be a category and $U$ a subobject in $\C$. Let $\bT_U$ be the locally constant fibred monad $\bT_C(X) = C \times U$. Then we have the following diagram of adjunctions
		\[
	\begin{tikzcd}[column sep= 2cm ]
		\Mnd(\C_{U}) \arrow[d, "P_U"]  \arrow[r, bend left=7, "\pi"] \arrow[r, leftarrow, bend right=7, "\bot", "P_{\bT_U}"']& \Mnd(\C) \arrow[d, "P"] \\ 
		\C_{U} \arrow[r, bend left=7, "\pi"] \arrow[r, leftarrow, bend right=7, "\bot", "P_U"'] & \C
	\end{tikzcd}.
	\]
\end{lemma}

\begin{proof}
	By \cref{lemma:pullback mnd}, the square with $\pi$ is a pullback square. Moreover, $\Mnd(\C)$ has a terminal object, given by the fibred monad $\bT_C(X) = C$. Hence, all desired implications follow from \cref{lemma:adj fcoswp}. 
\end{proof}

\begin{lemma} \label{lemma:mnd phi}
 Let $\C$ be a category and $\Phi$ a filter of subterminal objects. Let $F$ be a fibred monad on $\C$. Then $F$ induces a fibred monad on $\C_\Phi$, which we denote by $F_\Phi$.
\end{lemma}

\begin{proof}
  By \cref{lemma:mnd times u}, for a given $U$ in $\Phi$, we have $F_U(X \times U) = F(X) \times U$. Hence, by the property of colimits, this induces a functor $F_\Phi\colon \C_\Phi \to \C_\Phi$. It is now evident that it still has the structure and property of a monad. 
\end{proof}

\cref{lemma:mnd times u} implies that a category $\C$ and filter of subterminal objects $\Phi$ on $\C$, induces a filter of subterminal objects on $\Mnd(\C)$ given by $\bT_U$ for all $U$ in $\Phi$, which we also denote by $\Phi$. This directly results in the following definition.

\begin{definition} \label{def:mnd phi}
	Let $\C$ be a category and $\Phi$ a filter of subterminal objects. A $\Phi$-fibred monad is an object in $\Mnd(\C)_\Phi$.  
\end{definition}

Unwinding definitions, \cref{def:mnd phi} implies that a $\Phi$-fibred monad is a fibred monad on $\C_\Phi$, that is equal to $F_\Phi$ (in the sense of \cref{lemma:mnd phi}) for some fibred monad $F$ in $\Mnd(\C_U)$, for some $U$ in $\Phi$.

\begin{definition} \label{def:mnd wp phi}
	Let $\C$ be a category and $\Phi$ a filter of subterminal objects, $\cP$ a parameter. A $\Phi$-fibred monad  with parameters is a Cartesian functor $\Inst(\cP) \to \Mnd(\C)_\Phi$.  
\end{definition}

We can finally define cell monads (with parameters).

\begin{definition} \label{def:cell monad}
	Let $\cM$ be a model category with a class of cell monads, and $\Phi$ a model filter. A fibred monad on $\cM_\Phi$ is a \emph{cell monad} if it is equal to $\bT_\Phi$, for some cell monad $\bT$ on $\cM_U$, for some $U$ in $\Phi$. 
\end{definition}

\begin{definition} \label{def:cell monad wp}
	Let $\cM$ be a model category with a class of cell monads, $\Phi$ a filter of subterminal objects.	A \emph{cell monad with parameters} on $\cM_\Phi$ is a Cartesian functor $\bT\colon\cP \to \Mnd(\cM)_\Phi$ such that for all $\theta$ in $\cP$, the image is a cell monad on $\cM_\Phi$.
\end{definition}

\begin{example} \label{ex:cell monad wp excellent}
	Let $\cM$ be an excellent model category, and $\Phi$ a simplicial model filter. Then, following \cref{def:cell monad wp}, a cell monad with parameters on $\cM_\Phi$ is a Cartesian functor $\bT\colon\cP \to \Mnd(\cM)_\Phi$ such that for all $\theta$ in $\cP$, there exists a $U$ in $\Phi$, such that $\bT(\theta) = (\bT')_\Phi$, for some cell monad $\bT'$ on $\cM_U$, in the sense of \cite[Definition 11.10]{lumsdaineshulman2020goodexcellent}.
\end{example}

\begin{remark} \label{rem:cell monad}
 Note that $\cP$ is finite, hence if $\bT$ is cell monad with parameters, there exists a $V$ in $\Phi$, such that $(\bT_\theta)_V$ is a cell monad for all $\theta$.
\end{remark}

We can now state and prove the main result.

\begin{theorem} \label{thm:ia}
	Let $\cM$ be a good model category that satisfies \hyperlink{item:ia}{\textbf{IA}}, $\Phi$ a model filter. 
	\begin{enumerate}[leftmargin=*]
		\item Then $\cM_\Phi$ satisfies \hyperlink{item:ia}{\textbf{IA}} with respect to all cell monads with parameters, as defined in \cref{def:cell monad wp}.
		\item $(\cM_\Phi)_!$ has strictly stable typal initial $\bT_\Phi{-}\Alg_f$-algebras, with respect to all cell monads with parameters $\bT$ on $\cM_\Phi$.
		\item $P_\Phi\colon \cM \to \cM_\Phi$ maps weakly stable typal initial $\bT{-}\Alg_{\emph{\textbf{\text{f}}}}$-algebras to weakly stable typal initial $\bT_\Phi{-}\Alg_{\emph{\textbf{\text{f}}}}$-algebras.  
	\end{enumerate}
\end{theorem}
	
\begin{proof}
	Let $\bT$ be a cell monad with parameters. By \cref{rem:cell monad}, there is a $V$ such that for all $\theta$ in $\cP$, $(\bT_\theta)_V$ is a cell monad with parameters on $\cM_V$. By assumption $\cM_{V}$ satisfies \hyperlink{item:ia}{\textbf{IA}}, hence $(\cM_V)_{\textbf{\text{f}}}$ has weakly stable typal initial $\bT_V{-}\Algf$-algebras and $\bT_V{-}\Algf$ has representable lifts. $(1)$ and $(2)$ now follow from \cref{thm:filter fcoswp}. The last part follows from the fact that in the proof of \cref{thm:filter fcoswp} the weakly stable typal initial algebras are precisely constructed by projecting via $P_\Phi$.
\end{proof}

	% As $\cM_{V}$ satisfies \hyperlink{item:ia}{\textbf{IA}}, as a full subcategory of $\cM$. Hence, by assumption, $(\cM_V)_{\textbf{\text{f}}}$ has weakly stable typal initial $\bT_V{-}\Algf$-algebras and $\bT_V{-}\Algf$ has representable lifts. $(1)$ and $(2)$ now follow from \cref{thm:filter fcoswp}. The last part follows from the fact that in the proof of \cref{thm:filter fcoswp} the weakly stable typal initial algebras are precisely constructed by projecting via $P_\Phi$.
	
We can now apply this axiomatic result to the particular case of excellent model categories.

\begin{corollary} \label{cor:higher inductive types filter}
 Let $\cM$ be an excellent model category, $\Phi$ a simplicial model filter. Let $\bT_\Phi$ be a cell monad with parameters on $\cM_\Phi$ (in the sense of \cref{ex:cell monad wp excellent}). Then $(\cM_\Phi)_!$ has strictly stable typal initial $\bT_\Phi{-}\Alg_f$-algebras.
\end{corollary}
 
This establishes the main result we were aiming for. Notice, we established everything with a very axiomatic view on what a cell monad is. In the last part of this section, we want to analyze how this approach relates to cell monads in excellent model categories. Recall that a cell monad is generated via monad cells. We can see that monad cells are almost completely recovered in $\C_\Phi$. This requires reviewing the relation between endofunctors and monads. 

\begin{definition}
	Let $\C$ be a category and $F$ an endofunctor on $\C$. The \emph{monad generated by $F$} is a monad $\bT_F\colon \C \to \C$, satisfying the universal property that algebras over $\bT_F$ are equivalent to algebras over $F$.
\end{definition}

If the category $\C$ is locally presentable and the endofunctor fiber-wise accessible, then we can use a version of the small object argument to freely generate a fibred monad out of a fibred endofunctor \cite[Lemma 11.9]{lumsdaineshulman2020goodexcellent}, building on techniques in \cite{kelly1980transfinite}. As filter quotient categories are generally not locally presentable, we need a different approach. So, as a first step we study fibred endofunctors, fibred monads and their algebras over filter quotients.

For a given category $\C$, let $\End(\C) \to \C$ denote the Grothendieck fibration whose fiber over an object $C$ in $\C$ is given by the category of fibred endofunctors on $\C_{/C}$. The same proof as in \cref{lemma:mnd times u} gives us the following result.

\begin{lemma} \label{lemma:end times u}
	 Let $\C$ be a category and $U$ a subobject in $\C$. Let $\bT_U$ be the fibred endofunctor $\bT_C(X) = C \times U$. Then we have the following diagram of adjunctions
		\[
	\begin{tikzcd}[column sep= 2cm ]
		\End(\C_{U}) \arrow[d, "P_U"]  \arrow[r, bend left=7, "\pi"] \arrow[r, leftarrow, bend right=7, "\bot", "P_{\bT_U}"']& \End(\C) \arrow[d, "P"] \\ 
		\C_{U} \arrow[r, bend left=7, "\pi"] \arrow[r, leftarrow, bend right=7, "\bot", "P_U"'] & \C
	\end{tikzcd}.
	\]
\end{lemma}

Similar to \cref{lemma:mnd phi}, \cref{lemma:end times u} implies that every fibred endofunctor on $\C$	induces a fibred endofunctor on $\C_\Phi$, which we again denote $F_\Phi$. \cref{lemma:end times u} also implies that a filter of subterminal objects $\Phi$ on $\C$ induces a filter of subterminal objects on $\End(\C)$, which we also denote by $\Phi$, giving us the following definition.

\begin{definition} \label{def:end phi}
	Let $\C$ be a category and $\Phi$ a filter of subterminal objects. A $\Phi$-fibred endofunctor is an object in $\End(\C)_\Phi$.
\end{definition}

An object in $\End(\C)_\Phi$ is a fibred endofunctor on $\C_\Phi$, that is equal to $F_\Phi$ for some fibred endofunctor $F$ on $\C_U$ for some $U$ in $\Phi$. We now want to relate fibred endofunctors and free fibred monads.

\begin{lemma} \label{lemma:free monad}
	Let $F$ be an object in $\End(\C)_\Phi$, with lift $F_U$. Assume that $\bT_U$ is the free monad associated to $F_U$. Then the image of $\bT_U$ in $\Mnd(\C)_\Phi$ is the free monad associated to $F$. 
\end{lemma}

\begin{proof}
	By construction there is a natural transformation $\eta\colon F_U \to \bT_U$ and the restriction map $\eta^*\colon \bT_U{-}\Alg \to \F_U{-}\Alg$, that sends an algebra $\bT_UX \to X$ to the algebra $F_UX \xrightarrow{ \ \eta \ } \bT_UX \to X$, is an equivalence of categories. Applying $P_\Phi$ preserves equivalences of categories, meaning 
	\[[\eta]^*\colon \bT_\Phi{-}\Alg \to \F_\Phi{-}\Alg\] 
	is also an equivalence, where $[\eta]$ is the induced natural transformation $F_\Phi \to \bT_\Phi$.
\end{proof}
 
We now apply this result to monad cells on $\C_\Phi$. 

\begin{definition}
	Let $\C$ be a locally Cartesian closed category. A \emph{polynomial endofunctor} is a functor $Q\colon \C_{/C} \to \C_{/C}$ of the form $\C_{/C} \xrightarrow{ f^* } \C_{/A} \xrightarrow{g _*} \C_{/B} \xrightarrow{h_!} \C_{/C}$, for a given triple $(f,g,h) = C \xleftarrow{f} A \xrightarrow{g} B \xrightarrow{h} C$ in $\C$. 
\end{definition}

\begin{example} \label{ex:polynomial endofunctor}
	Let $\C$ be a simplicially enriched locally Cartesian closed category, $f\colon A \to B$ a morphism in $\C$ and $K$ a simplicial set. Then we obtain a fibred endofunctor $F^{f,K}$ on $\C$, which on $\C_{/C}$ is given by $F^f \otimes K$, where $F^f$ is the polynomial functor $(\pi_2\colon A \times C \to C, f \times C, \pi_2\colon B \times C \to C)$.
\end{example}

We now have the following fundamental fact about the fibred endofunctor $F^{f,K}$. 

\begin{lemma} \label{lemma:endomorphism lift}
	Let	$\C$ be a simplicially enriched locally Cartesian closed category, $\Phi$ a simplicial filter of subterminal objects, $f$ a morphism in $\C_\Phi$, and $K$ a finite simplicial set. Then the fibred endofunctor on $F^f\otimes K$ is in $\End(\C)_\Phi$.
\end{lemma}

\begin{proof}
	By definition, there exists a $U$ in $\Phi$, such that $f$ lifts to a morphism $\hat{f}$ in $\C_U$. By \cref{prop:filter quotient category}, the filter quotient construction preserves pullbacks, pushforwards, and simplicial tensor. Hence $(F^{\hat{f}} \otimes K)_\Phi = F^{f} \otimes K$. 
\end{proof}

\begin{proposition} \label{prop:free polynomial monad}
	Let $\C$ be locally presentable, locally Cartesian closed category and simplicially enriched, $\Phi$ a filter of subterminal objects, $f$ a morphism in $\C_\Phi$, and $K$ a finite simplicial set. Then the fibred endofunctor $F^f \otimes K$ on $\C_\Phi$ has an associated free fibred monad $\bT^{f,K}$ that is in $\Mnd(\C)_\Phi$. 
\end{proposition}

\begin{proof}
 By \cref{lemma:endomorphism lift}, $F^{f} \otimes K = (F^{\hat{f}} \otimes K)_\Phi$, for some $\hat{f}$ in $\C_U$. By \cite[Lemma 11.9]{lumsdaineshulman2020goodexcellent}, this generates a free fibred monad $\bT^{f,K}$ on $\C_U$.  Finally, by \cref{lemma:free monad}, $\bT^{f,K}$ is the free fibred monad associated to $F^{f} \otimes K$ on $\C_\Phi$, and hence by definition in $\Mnd(\C)_\Phi$.
\end{proof}

We now have the following immediate corollary, which gives a more concrete understanding of cell monads in filter quotient categories.

\begin{corollary}
	Let $\cM$ be an excellent model category, $\Phi$ a simplicial model filter, $f$ a fibration in $\cM_\Phi$, and $A \to B$ an inclusion of finite simplicial sets. Then the free monad cell $\bT^{f,A} \to \bT^{f,B}$ on $\cM_\Phi$ exists and is a cell monad on $\cM_\Phi$.
\end{corollary}

\begin{proof}
 By \cref{thm:filter quotient model structure}, there exists a fibration $\hat{f}$ in $\cM_U$ for some $U$ in $\Phi$, such that $P_\Phi(\hat{f}) = f$. The result now follows from the definition of cell monad \cite[Definition 11.10]{lumsdaineshulman2020goodexcellent} and \cref{prop:free polynomial monad}.
\end{proof}

\subsection{Univalent Universes} \label{subsec:universes}
In this subsection we prove that the filter quotient construction preserves univalent universes. This requires reviewing some aspects of its semantics, as presented in \cite{shulman2019inftytoposunivalent}. Intuitively a universe in a model category $\cM$ is a fibrant object $U$ in $\cM$, such that morphisms $X \to U$ are in an appropriate sense equivalent to fibrations $Y \twoheadrightarrow X$, functorial in $X$. Evidently, an object $U$ induces a representable functor $\cM(-,U)$, so as a first step we need to make the other side functorial.

\begin{definition} \label{def:bFib}
	Let $\cM$ be a finitely complete category. $\bM\colon\cM^{op} \to \Grpd$ is the pseudo-functor that takes $X$ in $\cM$ to $(\cM_{/X})^\simeq$, meaning the underlying groupoid of $\cM_{/X}$. If $\cM$ has additionally a model structure, we let $\bFib\colon\cM^{op} \to \Grpd$ be the full sub-functor of $\bM$ with objects fibrations in $\cM$.
\end{definition}

It would be unreasonable to expect an actual equivalence of groupoids between $\cM(-,U)$ and $\bM$, as it would imply $\bM$ takes values in sets. What we need is a more refined notion of equivalence that fits the situation. The solution is to transfer parts of the model structure from $\cM$ to the category of pseudo-functors $\Fun^{ps}(\cM^{op},\Grpd)$, as was originally realized by Shulman \cite[Definition 5.1]{shulman2019inftytoposunivalent}. 

\begin{definition} \label{def:acyclic fibration}
	Let $\cM$ be a model category and $\bX, \bY\colon \cM^{op} \to \Grpd$ pseudo-functors. A morphism $\alpha\colon \bX \to \bY$ is an \emph{acyclic fibration}, if it has the right lifting property with respect to $\cM(-,A) \to \cM(-,B)$ for all cofibrations $A \to B$ in $\cM$.
\end{definition}

We can use \cref{def:acyclic fibration} to define universes for an even broader class of pseudo-functors, as the one introduced in \cref{def:bFib}, namely \emph{locally representable and relatively acyclic notion of fibred structure covering all fibrations (LPaRANoFScaF)}. 

\begin{definition}[{\cite[Definition 3.1, Definition 3.10, Definition 5.11]{shulman2019inftytoposunivalent}}] \label{def:lparanofscaf}
	Let $\cM$ be a model category. A \emph{locally representable and relatively acyclic notion of fibred structure covering all fibrations (LPaRANoFScaF)} consists of the following data and conditions:
	\begin{itemize}[leftmargin=*]
		\item A pseudo-functor $\bF\colon \cM^{op} \to \Grpd$
		\item A discrete fibrations of pseudo-functors $\bF \to \bM$ with small fibers.
		\item For every morphism $\cM(-,Z) \to \bM$ the pseudo-pullback along $\bF \to \bM$ is representable.
		\item The image of $\bF$ in $\bM$, denoted $|\bF|$, is precisely $\bFib$.
		\item The resulting morphism $\bF \to \bFib$ is an acyclic	fibration.
	\end{itemize} 
\end{definition}

For such an $\bF$, intuitively, we would now hope a universe is an object $U$ along with a choice of acyclic fibration $\cM(-,U) \to \bF$. However, this is still unrealistic, as it would imply $U$ itself is classified by a map $1 \to U$, which results in a situation analogous to Russell's paradox. Here the solution is straightforward: we need to suitably restrict $\bF$ to smaller classes of fibrations. 

\begin{definition}
	Let $\cM$ be a cofibrantly generated model category. Let $\bF$ be a LPaRANoFScaF and $\kappa$ a regular cardinal. Let $\bF^\kappa$ be defined as the pullback $\bF \times_{\bM}\bM^{\kappa}$, where $\bM^\kappa$ is the full sub-pseudo-functor of $\bM$ given by $\kappa$-small morphisms. 
\end{definition}

With all these pieces at hand, Shulman finally proves that for a suitable model category $\cM$ (in particular it needs to be cofibrantly generated, with underlying category locally presentable) and $\bF$ a LPaRANoFScaF, there exists a regular cardinal $\lambda$, such that for all regular cardinals $\kappa$ larger than $\lambda$, the pseudo-functor $\bF^\kappa$ has a fibrant univalent universe $U_\kappa$ \cite[Theorem 5.22]{shulman2019inftytoposunivalent}.

Moving on to our situation, we want to repeat these arguments for filter quotients. Unfortunately, that is not possible, as filter quotients are generally not cofibrantly generated. On the other hand, this machinery is primarily used to generate a tower of universes. Hence, building on work done in \cite[Appendix A]{shulman2019inftytoposunivalent}, we will formulate a precise definition of such a tower of universes and show this property is indeed preserved by the filter quotient construction.

\begin{definition} \label{def:universe}
	Let $\cM$ be a model category and $\bF$ a LPaRANoFScaF. A \emph{universe} for $\bF$ is a choice of object $U$, along with a map of pseudo-functors $\cM(-,U) \to \bF$, such that the induced map onto the image is an acyclic fibration. An object in $\bF$ that is in the image of $U$ is a fibration \emph{classified} by $U$.
\end{definition}

Having defined universes, the only thing missing is univalence. Given a fibration $\tilde{U} \to U$ in a Cartesian closed model category $\cM$, Shulman constructs the \emph{universal object of equivalences} $\Eq(\tilde{U}) \to U \times U$, along with a section of the $\Delta\colon U \to U \times U$, denoted $\idtoequiv \colon U \to \Eq(\tilde{U})$. See \cite[Section 4]{shulman2015elegantunivalence} for a detailed construction. We can now use	this overview to define univalent universes.

\begin{definition} \label{def:univalent universe}
	Let $\cM$ be a model category and $\bF$ a LPaRANoFScaF. A \emph{univalent universe} for $\bF$ is a universe $U$ for $\bF$ such that the induced map $\idtoequiv\colon U \to \Eq(\tilde{U})$  is an equivalence in $\cM$, where the fibration $\tilde{U} \to U$ is the image of the identity map $\id_U\colon U \to U$	under the map $\cM(-,U) \to \bF$. 
\end{definition}

We now use the theory of pseudo-models, as discussed in \cite[Appendix A]{shulman2019inftytoposunivalent}, to develop a theory of universes suitable to our situation. Recall that a natural pseudo-model is a category $\cM$ with a terminal object and a representable strict discrete fibration $\omega\colon \Tm \to \Ty$ in $\Fun(\cM^{op},\Grpd)$ \cite[Definition A.1]{shulman2019inftytoposunivalent}. A \emph{natural model} is a pseudo-natural model such that $\Tm$ is discrete \cite{awodey2018natural}. The \emph{canonical natural pseudo-model} of a model category $\cM$ is given by $\bFib \times_{\bM} \bM_\bullet \to \bFib$, where $\bM_\bullet \to \bM$ is the forgetful fibration from pointed objects \cite[Example A.5]{shulman2019inftytoposunivalent}. We now use this background to define universes. 

\begin{definition}[Precise formulation of \hyperlink{item:u}{\textbf{U}}] \label{def:u}
 Let $\cM$ be a model category. $\cM$ satisfies \hyperlink{item:u}{\textbf{U}} if it has the following structure and properties:
	\begin{itemize}[leftmargin=*]
		\item $\cM$ is locally Cartesian closed. 
		\item $\cM$ has a LPaRANoFScaF denoted $\bF$ (\cref{def:lparanofscaf}).
		\item A \emph{level structure} $\cL$ \cite[Definition A.12]{shulman2019inftytoposunivalent}.
		\item For every $\alpha \in \cL$, a fibrant univalent universe $U_\alpha$ in $\cM$, in the sense of \cref{def:univalent universe}, such that $U_\alpha$ is strictly closed under $\Sigma$-types, $\Pi$-types, identity types, and binary sum types, in the sense of \cite[Definition A.16]{shulman2019inftytoposunivalent}, and containing the empty type, the unit type, the natural numbers type, the spheres types $S^n$, as well as other “cell complex” types such as the torus $T^2$.		
		\item The collection of $U_\alpha$ induce an $\cL$-family of strict Tarski universes in the canonical natural pseudo-model of $\cM$ \cite[Definition A.14]{shulman2019inftytoposunivalent} that classify all fibrations in $\cM$.
	\end{itemize}
\end{definition}

Using the ideas from \cite{lumsdainewarren2015localuniverses} and the work in \cite{lumsdaineshulman2020goodexcellent}, to every natural pseudo-model and choice of family of local universes $\cV$, in the sense of \cite[Definition A.19]{shulman2019inftytoposunivalent}, we can associate a natural model $\omega_!\colon \Tm_{!,\cV} \to \Ty_{!,\cV}$ which strictly models type theoretical constructors. If $\cM$ satisfies \hyperlink{item:u}{\textbf{U}}, then we can associate to the $\cL$-family of strict Tarski universes a family of local universes $\cV$ \cite[Example A.23]{shulman2019inftytoposunivalent}, resulting in a natural model $\omega_{!,\cV}\colon \Tm_{!,\cV} \to \Ty_{!,\cV}$, which we call the \emph{associated natural model}. Building on that, by \cite[Theorem A.25]{shulman2019inftytoposunivalent}, \hyperlink{item:u}{\textbf{U}} indeed suffices to model universes, meaning we have the following result.

\begin{corollary} \label{cor:u univalent universe}
	Let $\cM$ be a model category that satisfies	\hyperlink{item:u}{\textbf{U}}. Then the associated natural model $\omega_{!,\cV}\colon \Tm_{!,\cV} \to \Ty_{!,\cV}$ has strictly stable $\Sigma$-types, $\Pi$-types, identity types, and binary sum type, and an $\cL$-family of strict Tarski universes strictly closed under all these type constructors, meaning $\cM$ models arbitrarily large univalent universes.
\end{corollary}

We now want to prove that \hyperlink{item:u}{\textbf{U}} is preserved by the filter quotient construction. This requires several technical preliminary constructions and lemmas. Recall that to every pseudo-functor valued in groupoids denoted $\bF$, we can associate a Grothendieck fibration $\int_\C\bF \to \C$, via the Grothendieck construction. See \cite[Theorem 2.2.3]{loregianriehl2020fibration} for a detailed exposition.

\begin{definition} \label{def:pseudo-functor phi}
	Let $\C$ be a category with finite products, $\bF\colon \C^{op} \to \Grpd$ be a pseudo-functor, and $\Phi$ a filter of subterminal objects in $\C$. Then the \emph{induced pseudo-functor} $\bF_\Phi\colon \C_\Phi^{op} \to \Grpd$ is defined as the pseudo-functor associated to the Grothendieck fibration $(\int_{\C}\bF)_\Phi \to \C_\Phi$ (defined in \cref{lemma:filter quotient fibration}). 
\end{definition}

We will need the following basic observation regarding induced pseudo-functors, which is a direct computation, using the colimit description from \cref{lemma:filter quotient fibration}. 

\begin{lemma} \label{lemma:induced pseudo-functor}
 Let $\cM$ be a model category and $\Phi$ a model filter.
	\begin{itemize}[leftmargin=*]
		\item The induced pseudo-functor of $\bM$ is precisely $\bM_\Phi$, meaning the functor that maps an object $X$ in $\cM_\Phi$ to the groupoid $((\cM_{\Phi})_{/X})^\simeq$.
		\item The induced pseudo-functor of $\bFib$ is precisely $\bFib_\Phi$, meaning the functor that maps an object $X$ in $\cM_\Phi$ to the sub-groupoid of $((\cM_{\Phi})_{/X})^\simeq$ with objects fibrations.
		\item The induced	pseudo-functor of $\cM(-,X)$ is precisely $\cM_\Phi(-,X)$.
	\end{itemize}
\end{lemma}

\begin{lemma} \label{lemma:filter quotient trivial fibration}
	Let $\cM$ be a locally Cartesian closed model category, $\alpha\colon\bX \to \bY$ be a trivial fibration of pseudo-functors, and $\Phi$ a model filter. Then the natural transformation on induced pseudo-functors $\bX_\Phi \to \bY_\Phi$ is also a trivial fibration.
\end{lemma}

\begin{proof}
		We need to prove that $\bX_\Phi \to \bY_\Phi$ lifts against all maps of the form $\cM_\Phi(-,A) \to \cM_\Phi(-,B)$. Here $A \to B$ is a trivial cofibration in $\cM_\Phi$, meaning it is a trivial cofibration of the form $i \times U_0\colon A \to B$ in $\cM$, for some $U_0$ in $\Phi$. By the Yoneda lemma, a natural transformation $\alpha\colon\cM_\Phi(-,A) \to \bX_\Phi$ is uniquely determined by an object in the groupoid $\bX_\Phi(A) = \colim_{U \in \Phi} \bF_U(A)$, which by definition is an object $X$ in $\bX_{U_1}(A)$, for some $U_1$ in $\Phi$. Similarly, $\beta\colon \cM_\Phi(-,B) \to \bY_\Phi$ is uniquely determined by an object $Y$ in $\bY_{U_3}(B)$, for some $U_3$ in $\Phi$. Let $U \leq U_1,U_2,U_3$. Then, the diagram induced by $i, \alpha, \beta$ lifts to a diagram of pseudo-functors on $\cM_{U}$
		\[
		\begin{tikzcd}
			\cM_U(-,A)  \arrow[d, "i \times U"] \arrow[r, "\alpha"] & \bX_U \arrow[d] \\ 
			\cM_U(-,B) \arrow[r, "\beta"] \arrow[ur, dashed] & \bY_U
		\end{tikzcd}.
		\]
		By assumption, this diagram admits a lift $g\colon \cM_U(-,B) \to \bX_U$, which on filter quotients induces the desired lift $\cM_\Phi(-,B) \to \bX_\Phi$.
\end{proof}

\begin{proposition} \label{prop:induced lparanofscaf}
 Let $\cM$ be a locally Cartesian closed model category, $\bF$ a LPaRANoFScaF, and $\Phi$ a model filter. Then $\bF_\Phi$, defined via \cref{def:pseudo-functor phi}, is also a LPaRANoFScaF. 
\end{proposition}

\begin{proof}
	We check the conditions separately, relying on the computations established in \cref{lemma:induced pseudo-functor}.
	\begin{itemize}[leftmargin=*]
		\item We need to show the fiber of $\bF_\Phi \to \bM_\Phi$ is discrete and small. This follows from the fact that the fiber of $\bF_\Phi \to \bM_\Phi$ over an object $X$ is the filtered colimit of the fibers of $\bF_U \to \bM_U$, for $U$	in $\Phi$. Hence, the results from the fact that $\bF$ is a LPaRANoFScaF, and discrete small groupoids are closed under filtered colimits.
		\item Let $\alpha\colon\cM_\Phi(-,Z) \to \bM_\Phi$ be a natural transformation. By the Yoneda lemma, the data of such a natural transformation is equivalent to a choice of morphism $p\colon W \times U \to Z \times U$, for some $U$ in $\Phi$. This means we can lift $\alpha$ to a natural transformation $\alpha_U\colon \cM_U(-,Z) \to \bM_U$, with the property that $\alpha = \colim_{V \leq U} \alpha_V$, where $\alpha_V\colon\cM_V(-,Z) \to \bM_V$ is uniquely given by $p \times V$.
	 
		Now, we know that filtered colimits commute with pullbacks, and that the filter quotient of a representable functor is representable, by \cref{lemma:induced pseudo-functor}. Hence, to finish this step, it suffices to prove the pullback of the diagram   
		\[ 
		\begin{tikzcd}
			\cM_V(-,Z) \arrow[r] & \bM_V & \bF_V \arrow[l]
		\end{tikzcd}
		\]    
		is representable. However, this is true by assumption, as $\bF$ is a LPaRANoFScaF.
		\item By assumption, the map $\bF \to \bFib$ is surjective, and filtered colimits preserve surjectivity. Hence, the map $\bF_\Phi \to \bFib_\Phi$ is also surjective, meaning $|\bF_\Phi| = \bFib_\Phi$.
		\item This follows from applying \cref{lemma:filter quotient trivial fibration} to the trivial fibration $\bF \to \bFib$. \qedhere
	\end{itemize}
\end{proof}

\begin{proposition} \label{prop:univalent universe}
	Let $\cM$ be a locally Cartesian closed model category, $\Phi$ a model filter.
	\begin{enumerate}[leftmargin=*]
		\item $P_\Phi$ preserves universes. 
		\item $P_\Phi$ preserves the fibrancy of the universe. 
		\item $P_\Phi$ preserves the univalence of the universe.
		\item $P_\Phi$ preserves any type the universe contains, including empty type, unit type, the natural numbers type, the sphere types, and other ``cell complex types''.
	\end{enumerate}
\end{proposition}

\begin{proof}
	$(1)$ Let us denote the image of $\cM(-,U) \to \bF$ in $\bF$ by $|\cM(-,U)|$. By assumption, $\cM(-,U) \to |\cM(-,U)|$ is a trivial fibration of pseudo-functors, and so, by \cref{lemma:filter quotient trivial fibration}, $\cM_\Phi(-,U) \to |\cM_\Phi(-,U)|$ is also a trivial fibration, proving that $P_\Phi$ preserves universes.
	
	$(2)$ $P_\Phi$ preserves fibrancy, by \cref{thm:filter quotient model structure}. 

	$(3)$ Following \cite[Section 4]{shulman2015elegantunivalence} (see also \cref{def:univalent universe}), the universal object of equivalences $\Eq(\tilde{U}) \to U \times U$ is exclusively constructed via the locally Cartesian structure on $\cM$, which is preserved by $P_\Phi$ (\cref{prop:filter quotient category}). Hence, if $\idtoequiv\colon U \to \Eq(\tilde{U})$ is an equivalence, then it remains so after applying $P_\Phi$.
	
	$(4)$ By \cref{prop:filter quotient category}, $P_\Phi$ preserves the empty type, unit type, the natural numbers type, the spheres types, and other ``cell complex types'', as these are finite (co)limits. 
\end{proof}

	\begin{remark}
		We can gain an intuitive understanding of the fact that $P_\Phi$ preserves universes. Following \cite[Remark 5.4]{shulman2019inftytoposunivalent}, $U$ is a universe if and only if the following diagram admits a lift for a suitable choice of fibration $X \to B$.
	\[ 
		\begin{tikzcd}[row sep = .3cm]
			i^*(X) \arrow[rr] \arrow[dr] \arrow[dd, twoheadrightarrow] & & \tilde{U} \arrow[dd] \\ 
			 & X \arrow[ur, dashed] \arrow[dd, twoheadrightarrow]  & \\ 
			A \arrow[rr] \arrow[dr, "i"'] & & U \\ 
			& B \arrow[ur, dashed] & 
		\end{tikzcd}.
	\]
	We can still obtain such a lift in $\cM_\Phi$, as $P_\Phi$ preserves pullback squares (\cref{prop:filter quotient category}) and cofibrations (\cref{thm:filter quotient model structure}).
\end{remark}

\begin{proposition} \label{prop:closed universe}
	Let $\cM$ satisfy \hyperlink{item:u}{\textbf{U}} and $\Phi$ a model filter. If the $\cL$-family of pseudo Tarski universes is closed under $\Sigma$-types, $\Pi$-types, identity types, or binary sum types, then so is its image via $P_\Phi$. 
\end{proposition}

\begin{proof}
 We will consider the case of $\Pi$-types and the other cases follow similarly. By \cref{thm:filter quotient model structure}, $P_\Phi$ preserves fibrations, meaning it preserves both $\Ty$ (whose objects are fibrations) and $\Ty^\Pi$ (whose objects are pairs of fibrations, as defined in \cite[Below Remark A.8]{shulman2019inftytoposunivalent}). Moreover, $P_\Phi$ preserves surjectivity. Finally, by \cref{prop:univalent universe}, $P_\Phi$ preserves universes. Hence, applying $P_\Phi$ to the diagram in \cite[Definition A.16]{shulman2019inftytoposunivalent} for all $\alpha, \beta \in \cL$ we obtain the following diagram of functors from $\cM_\Phi^{op}$ to groupoids
	\[
	\begin{tikzcd}
	 \cM_\Phi(-,(\diamond\centerdot  P_\Phi U_\alpha) \triangleleft(\diamond\centerdot P_\Phi U_\beta)) \arrow[d] \arrow[dr, dashed] \arrow[rr, dotted] & & \cM_\Phi(-,\diamond\centerdot P_\Phi U_{\alpha\vee \beta}) \arrow[d] \\ 
		\Ty^\Pi & \bG \arrow[l, twoheadrightarrow] \arrow[r] & \Ty 
	\end{tikzcd}
	\]
 Given that the lift is preserved as part of the diagram, $P_\Phi$ preserves $\cL$-families of pseudo Tarski universes closed under $\Pi$-types (which simply means the lift to $\bG$ exists).
\end{proof}
	
\begin{theorem} \label{thm:u}
 Let $\cM$ be a model category, $\Phi$ a model filter. Then $P_\Phi\colon	\cM \to \cM_\Phi$ preserves \hyperlink{item:u}{\textbf{U}}. 
\end{theorem}

\begin{proof}
	First, $\cM_\Phi$ is locally Cartesian closed by \cref{prop:filter quotient category}. Next, by \cref{prop:induced lparanofscaf}, $P_\Phi$ maps the LPaRANoFScaF $\bF$ to the LPaRANoFScaF  $\bF_\Phi$. Next, we take the same level structure $\cL$ given for $\cM$. By \cref{prop:univalent universe}, $P_\Phi(U_\alpha)$ is a fibrant univalent universe. Moreover, by \cref{prop:closed universe}, the family of universes $P_\Phi(U_\alpha)$ is closed under the same type formers. Finally, by \cref{prop:filter quotient category}, $P_\Phi$ preserves pullback squares, and so in particular the pullbacks involved in the definition of strict Tarski universes. Hence, we are done. 
\end{proof}

We can now apply this general result to our case of interest. First of all we have the following lemma, stated as part of \cite[Theorem A.26]{shulman2019inftytoposunivalent}.

\begin{lemma} \label{lemma:ttmt u}
 Let $\cM$ be a type-theoretic model topos. Then $\cM$ satisfies \hyperlink{item:u}{\textbf{U}}. 
\end{lemma}

\begin{corollary} 
 Let $\cM$ be a type-theoretic model topos, $\Phi$ a simplicial model filter. Then $\cM_\Phi$ satisfies \hyperlink{item:u}{\textbf{U}}. Moreover, $P_\Phi$ preserves every universe that exists in $\cM$. 
\end{corollary}

\begin{proof}
	The first part follows from applying \cref{thm:u} to \cref{lemma:ttmt u}. The second part follows from \cref{prop:univalent universe}.
\end{proof}

\subsection{The Proofs} \label{subsec:proof}
We are now ready to combine all the previous steps and prove the two main results. 

\begin{proof}[Proof of \cref{thm:prop}]
	We analyze the properties separately.
	\begin{itemize}[leftmargin=*]
	\item	\hyperlink{item:flc}{\textbf{FLC}}, \hyperlink{item:et}{\textbf{ET}}, \hyperlink{item:lc}{\textbf{LC}}, \hyperlink{item:slc}{\textbf{SLC}}, \hyperlink{item:rp}{\textbf{RP}}: Follow directly from \cref{prop:filter quotient category}.  
	\item \hyperlink{item:s}{\textbf{S}}: Follows from \cref{thm:filter quotient model structure} and the assumption that $\Phi$ is simplicial. 
	\item \hyperlink{item:cim}{\textbf{CIM}}, \hyperlink{item:cem}{\textbf{CEM}}: By \cref{prop:filter quotient category}, $P_\Phi$ preserves monos, and by \cref{thm:filter quotient model structure} $P_\Phi$ preserves cofibrations. Hence every mono in $\cM_\Phi$ is either a cofibration for the first case, or precisely a cofibration, for the second case.
	\item \hyperlink{item:fe}{\textbf{FE}}, \hyperlink{item:tcp}{\textbf{TCP}}: By \cref{prop:filter quotient category}, $P_\Phi$ preserves exponentiability and pullbacks, and by \cref{thm:filter quotient model structure}, $P_\Phi$ preserves fibrations and trivial fibrations.
	\item \hyperlink{item:cl}{\textbf{CL}}: By \cref{prop:filter quotient category}, $P_\Phi$ preserves all limits in $\C_\Phi$, and by \cref{thm:filter quotient model structure}, $P_\Phi$ preserves cofibrations.
	\item \hyperlink{item:ia}{\textbf{IA}}: Proven in \cref{thm:ia}.
	\item \hyperlink{item:u}{\textbf{U}}: Proven in \cref{thm:u}.\qedhere
	\end{itemize}
\end{proof}

\begin{proof}[Proof of \cref{thm:main}] Throughout the proof assume $\cM$ is a model category and $\Phi$ a simplicial model filter. 
	\begin{enumerate}[leftmargin=*]
		\item Let $\cM$ be a logical model category, defined via the properties \hyperlink{item:flc}{\textbf{FLC}}, \hyperlink{item:tcp}{\textbf{TCP}}, and \hyperlink{item:fe}{\textbf{FE}} \cite[Definition 24]{arndtkapulkin2011modelstypetheory}. Then, by \cref{thm:prop}, $\cM_\Phi$ satisfies \hyperlink{item:flc}{\textbf{FLC}}, \hyperlink{item:tcp}{\textbf{TCP}}, and \hyperlink{item:fe}{\textbf{FE}}, meaning it is also a logical model category. Hence, $\cM_\Phi$ models the unit type, $\Sigma$-types and $\Pi$-types, by \cite[Theorem 26]{arndtkapulkin2011modelstypetheory}.
		\item Let $\cM$ be a type-theoretic model category, defined via the properties \hyperlink{item:lc}{\textbf{LC}}, \hyperlink{item:cl}{\textbf{CL}}, \hyperlink{item:rp}{\textbf{RP}}, and \hyperlink{item:fe}{\textbf{FE}}  \cite[Definition 2.12]{shulman2015homotopycanonicity}. Then, by \cref{thm:prop}, $\cM_\Phi$ satisfies \hyperlink{item:flc}{\textbf{FLC}}, \hyperlink{item:cl}{\textbf{CL}}, \hyperlink{item:rp}{\textbf{RP}}, and \hyperlink{item:fe}{\textbf{FE}}, meaning it satisfies all conditions of a type-theoretic model category except for infinite colimits. Following the proof of \cite[Proposition 2.13]{shulman2015homotopycanonicity}, it is still the case that its full subcategory	of fibrant objects is a type-theoretic fibration category \cite[Definition 2.1]{shulman2015homotopycanonicity}, hence still models identity types and function extensionality along with the type constructors from the previous item, by \cite[Section 4.2]{shulman2015homotopycanonicity}. 
		\item	Let $\cM$ be a good model category, defined via the properties \hyperlink{item:lc}{\textbf{LC}}, \hyperlink{item:s}{\textbf{S}}, \hyperlink{item:cim}{\textbf{CIM}}, \hyperlink{item:cl}{\textbf{CL}}, \hyperlink{item:rp}{\textbf{RP}}, and \hyperlink{item:slc}{\textbf{SLC}} \cite[Definition 2.1]{lumsdaineshulman2020goodexcellent}. Then, by \cref{thm:prop},  $\cM_\Phi$ satisfies \hyperlink{item:flc}{\textbf{FLC}}, \hyperlink{item:s}{\textbf{S}}, \hyperlink{item:cim}{\textbf{CIM}}, \hyperlink{item:cl}{\textbf{CL}}, \hyperlink{item:rp}{\textbf{RP}}, and \hyperlink{item:slc}{\textbf{SLC}}, meaning it satisfies all conditions of a good model category except for infinite colimits. Following the constructions and proofs in \cite[Section 3 - 6]{lumsdaineshulman2020goodexcellent}, the category of fibrant objects is still a type-theoretic fibration category which permits coproduct and pushout types, hence in particular includes various ``cell complex'' types, such as spheres and tori, along with all the type constructors from	the previous items. 
		\item Let $\cM$ be an excellent model category, defined via the properties \hyperlink{item:lc}{\textbf{LC}}, \hyperlink{item:s}{\textbf{S}}, \hyperlink{item:cim}{\textbf{CIM}}, \hyperlink{item:cl}{\textbf{CL}}, \hyperlink{item:rp}{\textbf{RP}}, \hyperlink{item:slc}{\textbf{SLC}}, \hyperlink{item:lp}{\textbf{LP}}, and \hyperlink{item:cg}{\textbf{CG}} \cite[Definition 2.1]{lumsdaineshulman2020goodexcellent}. Then, by \cite[Theorem 13.1]{lumsdaineshulman2020goodexcellent}, $\cM$ also satisfies \hyperlink{item:ia}{\textbf{IA}}. Thus, by \cref{thm:prop}, $\cM_\Phi$ satisfies \hyperlink{item:flc}{\textbf{FLC}}, \hyperlink{item:s}{\textbf{S}}, \hyperlink{item:cim}{\textbf{CIM}}, \hyperlink{item:cl}{\textbf{CL}}, \hyperlink{item:rp}{\textbf{RP}}, \hyperlink{item:slc}{\textbf{SLC}}, and \hyperlink{item:ia}{\textbf{IA}}. Finally, following the constructions and proofs in \cite[Section 3 - 8, 10]{lumsdaineshulman2020goodexcellent}, in addition to the previous constructors, the underlying category of fibrant objects of $\cM_\Phi$ also models various higher inductive types, such as a natural number type, $W$-types, propositional truncations, James constructions, and localizations.
		\item Let $\cM$ be a type-theoretic model topos, defined via the properties \hyperlink{item:lc}{\textbf{LC}}, \hyperlink{item:gt}{\textbf{GT}}, \hyperlink{item:rp}{\textbf{RP}}, \hyperlink{item:s}{\textbf{S}}, \hyperlink{item:cem}{\textbf{CEM}}, \hyperlink{item:lp}{\textbf{LP}}, \hyperlink{item:cg}{\textbf{CG}}, \hyperlink{item:slc}{\textbf{SLC}}, and \hyperlink{item:f}{\textbf{F}} \cite[Definition 6.1]{shulman2019inftytoposunivalent}. Then, $\cM$ also satisfies \hyperlink{item:et}{\textbf{ET}} (by \cite[Corollary III.7.4]{maclanemoerdijk1994topos}), \hyperlink{item:ia}{\textbf{IA}} (by \cite[Theorem 13.1]{lumsdaineshulman2020goodexcellent}) and \hyperlink{item:u}{\textbf{U}} (by \cref{lemma:ttmt u}). Thus, by \cref{thm:prop}, $\cM_\Phi$ satisfies \hyperlink{item:flc}{\textbf{FLC}}, \hyperlink{item:et}{\textbf{ET}}, \hyperlink{item:rp}{\textbf{RP}}, \hyperlink{item:s}{\textbf{S}}, \hyperlink{item:cem}{\textbf{CEM}}, \hyperlink{item:slc}{\textbf{SLC}}, \hyperlink{item:ia}{\textbf{IA}}, and \hyperlink{item:u}{\textbf{U}}. Hence, $\cM_\Phi$ models all previously mentioned type constructors, as well as, by \cref{cor:u univalent universe}, arbitrarily large strict univalent universes, in the sense of \cref{def:u}. \qedhere
	\end{enumerate}
\end{proof}

\bibliographystyle{alpha}
\bibliography{main}

\end{document}